\documentclass[12pt]{amsart}
\usepackage{a4wide,enumerate,color,graphicx}
\usepackage{amsmath}
\usepackage{scrextend} 
\usepackage{amsfonts}
\usepackage{amsthm}
\usepackage{mathrsfs}

\allowdisplaybreaks

\DeclareMathOperator{\divv }{div}

\DeclareMathOperator{\diff }{\delta}

\newcommand{\D}{{\rm {\small -}\kern - 6pt D}}
\newtheorem{theorem}{Theorem}
\newtheorem{lemma}[theorem]{Lemma}

\newtheorem{remark}[theorem]{Remark}

\newtheorem{definition}{Definition}

\begin{document}

\title[Gas mixtures with Maxwell-Stefan diffusion]{A theory of generalised solutions for ideal gas mixtures with Maxwell-Stefan diffusion}

\author[P.-E. Druet]{Pierre-Etienne Druet}
\address{Weierstrass Institute, Mohrenstr. 39, 10117 Berlin, Germany}
\email{pierre-etienne.druet@wias-berlin.de}

\date{\today}

\thanks{The first author was supported by the grant D1117/1-1 of the 
German Science Foundation (DFG)}

\begin{abstract}
After the pioneering work by Giovangigli on mathematics of multicomponent flows, several attempts were made to introduce global weak solutions for the PDEs describing the dynamics of fluid mixtures. While the incompressible case with constant density was enlighted well enough due to results by Chen and J\"ungel (isothermal case), or Marion and Temam, some open questions remain for the weak solution theory of gas mixtures with their corresponding equations of mixed parabolic--hyperbolic type. For instance, Mucha, Pokorny and Zatorska showed the possibility to stabilise the hyperbolic component by means of the Bresch-Desjardins technique and a regularisation of pressure preventing vacuum. The result by Dreyer, Druet, Gajewski and Guhlke avoids \emph{ex machina} stabilisations, but the mathematical assumption that the Onsager matrix is uniformly positive on certain subspaces leads, in the dilute limit, to infinite diffusion velocities which are not compatible with the Maxwell-Stefan form of diffusion fluxes. In this paper, we prove the existence of global weak solutions for isothermal and ideal compressible mixtures with natural diffusion. The main new tool is an asymptotic condition imposed at low pressure on the binary Maxwell-Stefan diffusivities, which compensates possibly extreme behaviour of weak solutions in the rarefied regime.  
\end{abstract}



\maketitle


\section{Introduction: The Maxwell-Stefan Navier-Stokes equations for an isothermal ideal fluid mixture}



The fluid  mixture here under consideration consists of $N\geq 2$ chemical species ${\rm A}_1,\ldots, {\rm A}_N$ subject to the following equations of mass and momentum balance
\begin{alignat}{2}
\label{mass}\partial_t \rho_i + \divv( \rho_i \, v + J^i) & = 0 & &  \text{ for } i = 1,\ldots,N\\
\label{momentum}\partial_t (\varrho \, v) + \divv( \varrho \, v\otimes v - \mathbb{S}) + \nabla p & = \sum_{i=1}^N \rho_i \, b^i(x, \, t)  & & \, .
\end{alignat}
The equations \eqref{mass} are the partial mass balances for the partial mass densities $\rho_1,\ldots,\rho_N$ of the species. We shall use the abbreviation $\varrho := \sum_{i=1}^N \rho_i$ for the total mass density. The barycentric velocity field of the fluid is called $v$ and the thermodynamic pressure is $p$. In the Navier-Stokes equations \eqref{momentum}, the viscous stress tensor is denoted $\mathbb{S}$. We overall consider Newtonian stress with constant coefficients and write $\mathbb{S} = \mathbb{S}(\nabla v)$. The vector fields $b^1,\ldots,b^N$ model the external body forces acting on each species per unit mass. The \emph{diffusion fluxes} $J^1,\ldots,J^N$ are defined to be the non-convective part of the mass fluxes, in view of which they obey the necessary side-condition $\sum_{i=1}^N J^i = 0$. 
One among the possibilities to close the model in thermodynamic consistent way, respecting also the constraint that the net diffusion flux vanishes, are the Maxwell-Stefan closure equations. In this approach, described a. o. by \cite{giovan}, \cite{bothedreyer}, \cite{bothedruetMS}, the diffusion fluxes $J^1,\ldots,J^N$ are defined implicitly as the solutions to a linear algebraic system with $N\times 3$ rows. For $i =1,\ldots,N$ the vector identities
\begin{align}\label{DIFFUSFLUX}
(\sum_{k \neq i} f_{i,k}\, y_k) \, J^i - y_i \, \sum_{k \neq i} f_{i,k} \, J^k = - \rho_i \, \sum_{k = 1}^N \, (\delta_{i,k} - y_k) \, (\nabla \mu_k - b^k)  \, 
\end{align}
are required to be valid in $\mathbb{R}^3$. For $1 \leq i < k \leq N$, the functions $f_{i,k}$ are given friction coefficients representing binary friction forces between the species; We define $f_{k,i} := f_{i,k}$ whenever $i > k$ (symmetric interaction). The friction coefficients $f_{i,k}$ are the inverse Maxwell-Stefan diffusivities $\D_{ik}$, but we shall avoid this heavier notation. We moreover assume that these coefficients are strictly positive functions of the state variables.
In \eqref{DIFFUSFLUX}, we write $\nabla \mu_k$ and $b^k$ instead of the correct $\nabla (\mu_k/\theta)$ and $b^k/\theta$ for the sake of visual convenience. This is justified as the temperature $\theta$ is overall assumed constant. We moreover denote $y = \hat{y}(\rho)$ the vector of the mass fractions: $y = (y_1, \ldots, y_N)$ with $y_i := \rho_i/\sum_{j=1}^N \rho_j =: \hat{y}_i(\rho)$ for $i=1,\ldots,N$. The quantities $d_i := \rho_i \, \sum_{k = 1}^N  (\delta_{i,k} - y_k) \, (\nabla \mu_k - b^k)$ on the right-hand side of \eqref{DIFFUSFLUX} are called the generalised driving forces for diffusion.
With the help of the matrix $B(\rho) = \{b_{i,k}(\rho)\}$ where 
\begin{align}\label{matrixB}
b_{i,k}(\rho) = - f_{i,k}(\rho) \, \hat{y}_i(\rho) \text{ for } i \neq k, \quad b_{i,i}(\rho) := \sum_{k \neq i} f_{i,k}(\rho)\, \hat{y}_k(\rho) \, , 
\end{align}
the equations \eqref{DIFFUSFLUX} can also be rephrased in the more compact way
\begin{align}\label{maxsteftensor}
 B(\rho) \, J = - R \, P(y) \, (\nabla \mu - b) \, ,
\end{align}
where $R = \text{diag}(\rho_1,\, \ldots,\rho_N)$, $P(y)$ is the projector $I - 1^N \otimes y$ and $1^N = (1,\ldots,1) \in \mathbb{R}^N$. 

In order to close the model \eqref{mass}, \eqref{momentum} with \eqref{DIFFUSFLUX}, we must specify the mass based chemical potentials $\mu_1,\ldots,\mu_N$ in their relationship to the state variables and the pressure. In this paper we restrict our considerations to so-called ideal mixtures. In this case the chemical potentials obey
\begin{align}\label{IDEALCHEMPOT}
\mu_i = \hat{\mu}_i(p, \, x_i) :=  g_i(p) + \frac{ k_B \, \theta}{m_i} \, \ln x_i \, .
\end{align}
The function $g_i(p)$ is given, it is the Gibbs-energy of constituent $i$ at temperature $\theta$. In other words, denoting $\hat{\rho}_i(\theta,\, p)$ the mass density of the $i$th constituent as function of temperature and pressure, then the derivative $g_i^{\prime}(p)$ is nothing else but $1/\hat{\rho}_i(\theta, \, p)$. 

Concerning the second contribution in \eqref{IDEALCHEMPOT}, the quantities $x_1,\ldots,x_N$ denote the number fractions. They are related to the mass densities via $x_i= := \rho_i/(m_i \, \sum_{j=1}^N (\rho_j/m_j)) =: \hat{x}_i(\rho) $. Here and throughout the paper, the positive constants $m_1, \ldots, m_N$ denote the elementary masses of the species.

In order to integrate the definition \eqref{IDEALCHEMPOT} into a thermodynamically consistent framework, we moreover have to provide the definition of $\mu$ as variational derivative of the Helmholtz free energy of the system. Here we shall rely on the assumption that the Helmholtz free energy possesses only a bulk contribution with density $\varrho\psi$. Moreover this function possesses the special form $ \varrho\psi = h(\rho_1, \ldots, \rho_N)$, where $h: \, \mathbb{R}^N_+ \rightarrow \mathbb{R}$ is convex and sufficiently smooth on the entire range of admissible mass densities. Then $\mu_1,\ldots,\mu_N$ are related to the mass densities $\rho_1, \ldots,\rho_N$ via
\begin{align}\label{CHEMPOT}
\mu_i = \partial_{\rho_i}h(\rho_1, \ldots, \rho_N) \, .
\end{align}
To make the link between, on the one hand, the free energy $h$ as function of the mass densities and, on the other hand, the prescribed chemical potentials \eqref{IDEALCHEMPOT}, we use the Euler equation that connects all state variables
\begin{align}\label{GIBBSDUHEMEULER}
p = -h(\rho_1,\ldots,\rho_N) + \sum_{i=1}^N \rho_i \, \mu_i = -h(\rho_1,\ldots,\rho_N) + \sum_{i=1}^N \rho_i\, \partial_{\rho_i}h(\rho_1, \ldots, \rho_N) \, 
\end{align}
where for simplicity $\theta$ does not occur, as it is a constant. We solve \eqref{GIBBSDUHEMEULER} with the integration method proposed in the paper \cite{bothedruetfe}. We first define a function of pressure and mass densities via $V(p, \,\rho) := \sum_{i=1}^N g_i^{\prime}(p) \, \rho_i$. The dimensionless function $V(p, \,\rho)$ acts as a 'volume potential' and it describes the contribution of volume extension to the free energy density. Then the resulting pressure state equation $p = \hat{p}(\rho)$ is implicitly obtained from the identity
\begin{align}\label{PRESSURESTATE}
V(\hat{p}(\rho), \,\rho) = \sum_{i=1}^N g_i^{\prime}(\hat{p}(\rho)) \, \rho_i = 1 \, .
\end{align}
Under the assumption - to be made more precise below - that the functions $p \mapsto g_i(p)$ are sufficiently smooth, concave and non-decreasing, the implicit function $\rho \mapsto \hat{p}(\rho)$ is uniquely determined by \eqref{PRESSURESTATE} and sufficiently smooth. The Helmholtz free energy density $h$ compatible with \eqref{IDEALCHEMPOT} and \eqref{CHEMPOT} assumes the form
\begin{align}\label{FE}
 h(\rho) = \sum_{i=1}^N \rho_i \, \Big(g_i(\hat{p}(\rho)) + \frac{C(\theta)}{m_i} \, \ln \hat{x}_i(\rho) + c_i(\theta)\Big) - \hat{p}(\rho) \, .
\end{align}
Up to the integration constants $C(\theta)$ and $c_1(\theta),\ldots,c_N(\theta)$, the function $h$ is likewise uniquely determined by the definition \eqref{CHEMPOT}.

As an illustration of these concepts and definitions, consider the case of all $g_i$ are logarithmic, more precisely in the form 
\begin{align}\label{giexample}
g_i(p) = g_i^0 + p_0 \, \bar{v}_i^0 \, \ln (p/p_0) \, , 
\end{align}
where $g^0 \in \mathbb{R}^N$ are reference values, $\bar{v}^0_1, \ldots, \bar{v}^0_N > 0$ are reference volumes per unit mass, and $p_0 > 0$ is a reference pressure. Then, inversion of \eqref{PRESSURESTATE} yields a state equation $\hat{p}(\rho) := p_0 \, \sum_{i=1}^N \bar{v}_i^0 \, \rho_i$, and the associated form of the free energy is easily gained from \eqref{FE}. 

Let us remark in passing that we call a mixture ideal if it is volume additive -- see for instance \cite{brdicka}, section 4.1.2. If the chemical potentials are subject to \eqref{IDEALCHEMPOT}, the specific mixture volume is $\bar{v} = \sum_{i=1}^N g_i^{\prime}(p) \, y_i$, and it is nothing else but the sum of the specific volumes of the constituents weighted by their relative contributions to the total mass. Reversely, it can be shown in rigorous mathematical framework (cf. \cite{bothedruetfe}) that the assumption of volume additivity generates the splitting \eqref{IDEALCHEMPOT} of the chemical potentials, and that the only compatible form of the Helmholtz potential is given by \eqref{FE}.

\textbf{The initial-boundary-value problem.} The universal equations \eqref{mass}, \eqref{momentum} supplemented by the Maxwell-Stefan closure equations \eqref{DIFFUSFLUX}, in which $\mu$ and $p$ are defined via \eqref{IDEALCHEMPOT} and \eqref{PRESSURESTATE}, now form a closed system of partial differential equations for the densities $\rho_1,\ldots,\rho_N$ and the velocities $v_1, \, v_2,\, v_3$. We call it the Maxwell-Stefan Navier-Stokes system for an isothermal ideal compressible fluid mixture.

In this paper, we shall investigate the resolvability of typical initial-boundary-value problems for this PDE-system in a cylindrical domain $Q_T := \Omega \times ]0,T[$ with a bounded, open, connected cross-section $\Omega \subset \mathbb{R}^3$ and $T>0$ a finite time. 
We consider the initial conditions
\begin{alignat}{2}\label{initial0rho}
\rho_i(x,\, 0) & = \rho^0_i(x) & & \text{ for } x \in \Omega, \, i = 1,\ldots, N \, ,\\
\label{initial0v} v_j(x, \, 0) & = v^0_j(x) & & \text{ for } x \in \Omega, \, j = 1,2,3\, .
\end{alignat}
For simplicity, we consider only linear homogeneous boundary conditions on the lateral surface $S_T := \partial \Omega \times ]0,T[$
\begin{alignat}{2}
\label{lateral0v} v & = 0 & &\text{ on } S_T \\
\label{lateral0q} \nu \cdot J^i & = 0 & &\text{ on } S_T \text{ for } i = 1,\ldots,N \, .
\end{alignat}
\begin{definition}\label{ProblemP}
The initial--boundary--value problem consisting of the Maxwell-Stefan Navier-Stokes system \eqref{mass}, \eqref{momentum}, \eqref{DIFFUSFLUX} on $Q_T$ with the initial conditions \eqref{initial0rho}, \eqref{initial0v} on $\{0\} \times \Omega$ and \eqref{lateral0v}, \eqref{lateral0q} on $S_T$ is denoted $(P)$.
\end{definition}

\textbf{Some notations.} In this paper, vectors of mass densities $\rho$ live in the state space $$\mathbb{R}^N_+ := \{\rho \in \mathbb{R}^N \, : \, \rho_i > 0 \text{ for } i =1,\ldots,N\} \, .$$

Vectors of mass fractions $y$ and number fractions $x$ live on the hyper-surface $$S^1_+ := \{y \in \mathbb{R}^N_+ \, : \, \sum_{i=1}^N y_i = 1\} \, .$$ 

Given a state $\rho \in \mathbb{R}^N_+$, the associated mass and number fractions are given as
\begin{align*}
y_i = \hat{y}_i(\rho) := \frac{\rho_i}{\sum_{j=1}^N \rho_j}, \quad  x_i = \hat{x}_i(\rho) := \frac{\rho_i}{m_i \, \hat{n}(\rho)}, \quad \hat{n}(\rho) := \sum_{j=1}^N \frac{\rho_j}{m_j} \, .
\end{align*}
We shall encounter several tensors of size $N\times 3$ like for instance the diffusion fluxes $J$ or the body forces $b$. For such rectangular matrices, the line index is the upward one, and the column index the downward one. 

Scalar products are denoted $ \cdot$ in $\mathbb{R}^N$ and $\mathbb{R}^3$ indifferently. For two $m \times n$ matrices $A, \, B$, we denote $A \, : \, B = \sum_{i=1}^m \sum_{k=1}^n A_{i,k} \, B_{i,k}$.

For a differentiable function $f$ of the state variables, we denote $\partial_{\rho_i}f$, or alternatively $f_{\rho_i}$ the partial derivatives, and $\nabla_{\rho} f = (\partial_{\rho_1}f,\ldots,\partial_{\rho_N} f) = f_{\rho}$ is the gradient.

If the vector $\rho$ is specified by the context, we denote $R = R(\rho)$ the diagonal matrix $\text{diag}(\rho_1,\, \ldots,\rho_N)$, and $P(y) = P(\hat{y}(\rho))$ is the barycentric projector $I - 1^N \otimes y$ on $\mathbb{R}^N$.

We denote $1^N = (1, \ldots, 1) \in \mathbb{R}^N$, and $\{1^N\}^{\perp} \subset \mathbb{R}^N$ is the orthogonal complement of the span of $1^N$. The orthogonal projection $I - \frac{1}{N} \, 1^N \otimes 1^N$ is denoted $\mathcal{P}_{\{1^N\}^{\perp}}$.

In the identity \eqref{FE} we shall choose $C(\theta) = 1$ and $c_i(\theta) = 0$, which is of no consequence for the mathematical analysis.

We denote $p_0$ the chosen reference pressure, $\bar{v}^0 $ is the vector $(g^{\prime}_1(p_0), \ldots, g^{\prime}_N(p_0))$ and $S_0$ the reference hyper-surface $ S_0 := \{\rho_i \in \mathbb{R}^N_+ \, : \, \bar{v}^0 \cdot \rho = 1 \}$.

\section{The main result}\label{mathint}

\subsection{State of the art}\label{STATEART}

While the modelling of mass, momentum and energy transport in multicomponent reactive fluids is probably as old as thermodynamics, the mathematical investigation of well posedness issues concerning these models is rather recent. In the very comprehensive book \cite{giovan}, Giovangigli was the first to develop a systematic view on modelling \emph{and} analysis. In the chapters 8 and 9, the fundamental underlying hyperbolic--parabolic structure of multicomponent flows with Maxwell-Stefan diffusion is first exhibited. To reveal this structure he resorts to equivalent so-called entropic variables, a tool which had been introduced and developed before in the context of hyperbolic conservation laws: See \cite{giovan}, Ch. 8 for details. In the case of smooth coefficients, these techniques allowed in \cite{giovan}, Th. 9.4.1 to establish global--in--time existence results for strong solutions to the Cauchy problem, if the initial conditions are close enough to a constant equilibrium state.  
These breakthrough results in principle gave a positive answer to the theoretical questions of well-posedness for multicomponent flow models, but they do not exhaust the field. Indeed, they rely on several abstract assumptions concerning the transformation to entropic variables, and also on the assumption that there is a constant equilibrium. In particular, the weak solution analysis is entirely left open. 

Concerning the theory of strong solutions, several authors looked later at different aspects of the multicomponent flow models. Maxwell-Stefan diffusion in a much more reduced setting, without mechanics and energy balance, was for instance investigated in \cite{herbergpruess}, while the incompressible Maxwell-Stefan-Navier-Stokes case with $\varrho = const$ was studied in \cite{bothepruess}. In both cases the hyperbolic component is avoided due to the assumptions of mechanical balance or incompressibility. In these papers, the PDE system is investigated in the original variables $\rho_1, \ldots,\rho_N$, and local--in--time well-posedness is proved in classes of maximal parabolic regularity. 
In the investigation \cite{bothedruet} on strong solutions, we investigate the well-posedness in classes of strong solutions ($L^p-$setting) models similar to the one of the book \cite{giovan} (for the isothermal case, but more general free energy descriptions are included). By means of a similar passage to entropic variables in order to isolate the hyperbolic component, we prove the local--in--time well--posedness for initial-boundary-value problems. We do recover also the global existence result if the initial condition is sufficiently near to a stationary \emph{non-constant} solution, but only for a finite time-horizon. Our estimates yield stability of the hyperbolic variable $\rho$ only for smooth equilibrium solution and initial data. Another recent investigation worth mentioning on strong solutions is \cite{piashiba18}.

On the weak solution side, not only the mixed parabolic--hyperbolic character of the PDE system is the source of tough problems. The question of positivity and/or species vanishing is also a major issue. In the approach of strong or classical solutions, this problem is usually handled using the short-time vicinity to the initial data, the global vicinity of the equilibrium solution (cf. \cite{giovan}, Ch. 9 and \cite{bothedruet}) or using the maximum principle for smooth solutions (cf. \cite{giovan}, 7.3.3, \cite{herbergpruess}, \cite{bothedruetMS}, section 6). Moreover, relaxing the regularity of the velocity field in the Navier-Stokes equations implies that the solution to the continuity equation is not any more bounded away from zero and infinity: complete vacuum and blow-up of the total mass density are phenomena that we must expect. Thus, the main challenges for a theory of weak solutions are 
\begin{itemize}
\item The mixed parabolic--hyperbolic character;
\item The dilute limit (vanishing constituents);
\item Extreme behaviour of the total mass density (vacuum and blow-up).
\end{itemize}
To get rid of a part of these difficulties, the incompressibility assumption was used in \cite{chenjuengel} to establish the global--in--time existence of some weak solutions. Interestingly, the passage to entropic variables was used again in this investigation, showing that this reformulation provides a natural framework both for the weak and the strong solution analysis. Very similar results are obtained in \cite{mariontemam} published the same year. 
While the assumption of incompressibility or mechanical equilibrium oblige in \cite{chenjuengel} and \cite{mariontemam} to some readjustments of the thermodynamics, the full models exposed in \cite{giovan} were investigated from the viewpoint of weak solutions by \cite{mupoza15}. In this investigation, the PDE system is handled with the help of several steps of stabilisation/regularisation: A pressure blow-up at density zero is admitted to exclude vacuum, while the hyperbolicity of the total mass density is tempered by means of the so-called Bresch-Desjardins device. A special structure of the viscous stress tensor yields some bounds on the density gradient, so that all partial mass densities, and with them the pressure, are weakly differentiable in space. 

Afterwards, coming from the different context of electrochemistry, the team of authors of \cite{dredrugagu17a}, \cite{dredrugagu17b}, \cite{dredrugagu17c} proposed a weak existence theory for equations describing the multicomponent flow of charged constituents. Accordingly, the underlying thermodynamic model relies on the structure \eqref{IDEALCHEMPOT} which is slightly more general than the models in \cite{giovan}. In this context it is more natural to admit that the pressure of the constituents is growing super linearly, so that it was possible to prove the global existence of certain weak solutions employing the Lions compactness method -- thus, with almost the same methods as for the single component fluids. In \cite{dredrugagu17a}, we also employed a change of variables similar to \cite{giovan} to isolate the hyperbolic density component. In one point however, the analysis remains unsatisfactory. We assume there Fick-Onsager diffusion fluxes of the form $J = - M(\rho) \, (\nabla \mu - b)$ -- which in general is equivalent to Maxwell-Stefan diffusion -- but have to assume \emph{moreover} that the Onsager operator is \emph{uniformly strictly positive} on the orthogonal complement of $1^N $. This provides a uniform control on the gradients of the entropic variables. A drawback of this assumption is that the Maxwell--Stefan diffusion case is not any more covered in the dilute limit (vanishing constituents), as infinite diffusion velocities must occur. 

In order to obtain a result including Maxwell--Stefan diffusion, the author relaxed in the subsequent paper \cite{dru16} the necessary assumptions for the Onsager operator. The correct dilute behaviour is recovered in the range of finite, positive mass densities, and it degenerates only on the vacuum. Unfortunately, this result is very technical, it relies on the Fick-Onsager formulation and, due to the additional features coming from the electrochemistry model, the preprint does not convey a clear and direct message.

From this brief review of the literature, we conclude that there is no published result on global weak solutions to the equations of multicomponent Maxwell-Stefan diffusion in a compressible fluid, as far as it avoids both stabilisations and/or the assumption that the equivalent Fick-Onsager form is uniformly positive on $\{1^N\}^{\perp}$. The aim of this paper is to provide such a result in the simplest possible setting where only Maxwell-Stefan diffusion and the isothermal multicomponent compressible Navier-Stokes equations are taken into account, with the ideal form \eqref{IDEALCHEMPOT} of the chemical potentials. Compared with the previous investigations, the existence of global solutions shall be enforced here by means of two types of modelling assumptions:
\begin{enumerate}[(a)]
 \item At high densities, the pressure of the single constituents is a sufficiently fast super linear function of their density;
 \item Maxwell-Stefan frictions coefficients $f_{i,k}$ are regular and strictly positive at finite pressure, but they might exhibit at low pressure a singular dependence on the state (cp.\ the formula \eqref{gbitteexplicit}).
\end{enumerate}

\subsection{The growth conditions and the main result}\label{GROWTH}

As temperature is a constant, we reinterpret the functions $g_i$ occurring in \eqref{IDEALCHEMPOT} as functions of pressure only. For $i = 1,\ldots,N$ the function $g_i: \, ]0, \, + \infty[ \rightarrow \mathbb{R}$ is subject to the following assumptions (A): 
\begin{labeling}{(A44)}
\item[(A1)] $g_i \in C^2(]0, \, +\infty[)$;
\item[(A2)] $g_i^{\prime}(p) > 0$ and $ g_i^{\prime\prime}(p) < 0$ for all $p > 0$; 
\item[(A3)] $\lim_{p \rightarrow 0} g^{\prime}_i(p) = + \infty$ and $\lim_{p \rightarrow +\infty} g^{\prime}_i(p) = 0$;
\end{labeling}
The conditions (A2) express for each constituent the requirement that: first its density $\hat{\rho}_i(p) = 1/g_i^{\prime}(p)$ is positive (first condition), and second that its volume $g_i^{\prime}(p)$ is a strictly decreasing function on pressure (second condition). The first of the conditions in (A3) means that the lower-pressure threshold corresponds to infinite volume, and the second one that infinite pressure leads to volume zero.
Up to now the assumptions on $g_i$ are obvious. For the mathematical analysis, we shall moreover require specific asymptotic growth assumptions for small and large arguments:
\begin{labeling}{(A44)}
\item[(A4)] For small and moderate arguments, we assume that $g_i$ behaves like $\bar{c}_i \, \ln p $, more precisely that there are a positive constants $c_1 < c_2$ and $M_0$ such that
\begin{align*}
0< c_1 \leq g_i^{\prime}(p) \, p \leq c_2 \text{ for all } 0 < p < M_0 \, .
\end{align*}
\item[(A5)] For large arguments, we assume that $g_i$ behaves like $\bar{c}_i \, p^{\frac{1}{\alpha_i}}$ with $\alpha_i > 1$, more precisely that there are $\alpha_1, \ldots, \alpha_N \geq \beta > 1$ and $M_1 > 0$ such that
\begin{align*}
 \alpha_i \, p \, g_i^{\prime}(p) \geq g_i(p) \geq \beta \, p \, g_i^{\prime}(p) \text{ for all } p \geq M_1 \, .
\end{align*}
\end{labeling}
Concerning the friction coefficients $f_{i,k}$ occurring in the Maxwell-Stefan equations \eqref{DIFFUSFLUX} we assume that they are given as functions of the state variables expressed by the pressure $p$ and the composition vector $x$ (number fractions). The domain of definition of composition vectors is the positive unit-sphere of the one-norm $S^1_+$,
and we denote $\overline{S}^1_+$ the closure in $\mathbb{R}^N$. For the matrix of coefficients $\{f_{i,k}\}$, we make for all $i < k$ the following natural modelling assumptions (B):
\begin{labeling}{(B44)}
\item[(B1)] $f_{i,k} \in C(]0, \, +\infty[ \times \overline{S}^1_+)$;
\item[(B2)] $f_{i,k}(p, \, x) > 0$ for all $p \in ]0, \, +\infty[$ and all $x \in \overline{S}^1_+$.
\end{labeling}
In order to formulate the specific asymptotic conditions, we distinguish between a low-pressure range $p < p_1$ and a normal pressure range $p > p_1$ for some $p_1 > 0$. We assume that
\begin{labeling}{(B44)}
\item[(B3)] There are constants $0< f_0 \leq f_1 < + \infty$ such that for the singular functions $\sigma_{i,k}$ defined just here below in \eqref{frefbitte} 
\begin{align*}
 f_0 \, \sigma_{i,k}(p, \, x)  \leq  f_{i,k}(p, \, x) \leq f_1 \, \sigma_{i,k}(p, \, x) \text{ for all } (p,x) \in ]0, \, p_1[ \times \overline{S}^1_+ \, .
 \end{align*}
\item[(B4)] There are constants $0< f_2 \leq f_3 < + \infty$ such that $ f_2 \leq  f_{i,k}(p, \, x) \leq f_3$ in $]p_1, \, +\infty[ \times \overline{S}^1_+$;
\end{labeling}
Recalling that $m_1, \ldots,m_N$ are the elementary masses of the species and $p_0$ the reference pressure, the general form of the coefficients $\sigma_{i,k}$ in (B3) is given by the formula
\begin{align}\label{frefbitte}
 \sigma_{i,k}(p, \, x) := & \sigma_0(p, \, x) \,  \frac{\exp(m_i \, (g_i(p_0)-g_i(p))) \, \exp(m_k \, (g_k(p_0)-g_k(p)))}{\hat{\chi}^{m_{i}+m_k}} \, ,\\[0.1cm]
  \sigma_0(p, \, x) := & \frac{\sum_{\ell = 1}^N m_{\ell} \, x_{\ell}}{\sum_{\ell = 1}^N g_{\ell}^{\prime}(p) \, m_{\ell} \, x_{\ell}} \, \frac{\sum_{\ell = 1}^N g^{\prime}_{\ell}(p_0) \, m_{\ell} \, x_{\ell} \,  \exp(-m_\ell \, (g_{\ell}(p_0)-g_{\ell}(p))) \, \hat{\chi}^{m_{\ell}}(p, \, x)}{\sum_{\ell = 1}^N m_{\ell} \, x_{\ell} \, \exp(m_\ell \, (g_{\ell}(p_0)-g_{\ell}(p)))  \, \hat{\chi}^{-m_\ell}} \nonumber\, ,
 \end{align}
 in which the positive factor $\hat{\chi}(p, \, x)$ is defined implicitly as the root of the equation $$\sum_{j=1}^N \hat{\chi}^{m_j} \, x_j \, \exp(-m_j \, (g_j(p_0)-g_j(p))) = 1 \, .$$

 In order to interpret \eqref{frefbitte}, we choose identical masses $m_1 = m_2 = \ldots = m_N =: m$. This allows to eliminate the inexplicit factor $\hat{\chi}$, and we obtain that
\begin{align}\label{fref2}
\sigma_{i,k}(p, \, x) = \frac{\sum_{\ell=1}^N g^{\prime}_{\ell}(p_0) \, x_{\ell} \,   e^{-m \, (g_{\ell}(p_0)-g_{\ell}(p))}}{ (\sum_{\ell = 1}^N g_{\ell}^{\prime}(p) \, x_{\ell}) \, (\sum_{\ell = 1}^N  x_{\ell} \, e^{m \, (g_{\ell}(p_0)-g_{\ell}(p))})} \, e^{m \, (g_i(p_0)-g_i(p))} \, e^{m \, (g_{k}(p_0)-g_{k}(p))} \, .  
\end{align}
Insert the special choice $g_i(p) = \bar{c}_i \, \ln(p/p_0)$ to find that the functions
\begin{align}\label{gbitteexplicit}
 \sigma_{i,k}(p, \, x) = \frac{\sum_{\ell=1}^N \bar{c}_{\ell} \, x_{\ell} \,  (\frac{p}{p_0})^{m\bar{c}_{\ell}}}{ \, (\sum_{\ell = 1}^N \bar{c}_{\ell} \, x_{\ell})\, (\sum_{\ell = 1}^N  x_{\ell} \, (\frac{p_0}{p})^{m\bar{c}_{\ell}})} \,  \left(\frac{p_0}{p}\right)^{m \, (\bar{c}_{i}+\bar{c}_k) - 1} \, .
\end{align}
define the asymptotic behaviour for $p \rightarrow 0$.

The conditions (B3), (B4) can therefore be interpreted as follows. For finite pressures, or equivalently, for finite total mass density, the friction coefficients $f_{i,k}$ (the inverse Maxwell-Stefan diffusivities) are bounded and positive. For extreme behaviour of the density (rarefied regime), the friction coefficients are singular and depend on the composition and the reference values for mass $m$, pressure $p_0$ and partial volumes $\bar{c}/p_0$ in a complex manner.
In this paper, we do not attempt to justify the condition (B3) from physical considerations. Let us merely notice that the possibilities to measure diffusivities for multicomponent systems are generally limited, not to mention approaching extreme states like vacuum. As to molecular dynamical simulations, they for the most rely on isobaric models, and are for this reason more suited to capture the compositional dependence of the $f_{i,k}$ than the dependence on parameters like pressure or total density. We refer to \cite{bothedruetMS} for details. Hence, keeping of course all cautions that are necessary in making such statements, we point at the fact that a mathematically motivated condition like (B3) might also serve as a guideline to meaningfully extrapolate known models for friction coefficients.

\subsection{Main result}\label{MAINRES}

Assume that $\Omega \subset \mathbb{R}^3$ is bounded and Lipschitz. For $T > 0$, denote $Q_T = \Omega \times ]0, \, T[$ and $S_T = \partial \Omega \times ]0,T[$. 
We shall use standard parabolic Lebesgue spaces $L^{p,q}(Q_T)$ with $1 \leq p,q \leq + \infty$. To vary a little bit, we use the original notations of the monograph \cite{ladu}, putting the space integration index first. Recall that $L^{p,p}(Q_T) = L^p(Q_T)$. We define $W^{1,0}_p(Q_T)$ the subspace of $L^1(Q_T)$ consisting of functions which possess all weak spatial derivatives in $L^p(Q_T)$. Vector valued spaces are denoted $L^{p}(Q_T; \, \mathbb{R}^k)$, $W^{1,0}_p(Q_T; \, \mathbb{R}^k)$, where $k \geq 1$ is an integer.\\[1ex]

A weak solution to the i.-b.-v. problem $(P)$ consists of
\begin{enumerate}[(i)]
\item \label{massdensities} A vector of mass densities $\rho \in L^{\gamma,\infty}(Q_T; \, \mathbb{R}^N)$ assuming nonnegative values almost everywhere in $Q_T$. Here $\gamma > 1$ is some constant fixed by the data;
\item \label{velocities} A velocity field $v \in W^{1,0}_2(Q_T;\,  \mathbb{R}^3)$ such that with $\varrho := \sum_{i=1}^N \rho_i$, the field $\sqrt{\varrho} \, v$ belongs to $L^{2,\infty}(Q_T; \, \mathbb{R}^3)$;
\item \label{flux} A diffusion flux matrix $J \in L^{\frac{2\gamma}{1+\gamma},2}( Q_T;\, \mathbb{R}^{N\times 3})$ assuming values in $\{1^N\}^{\perp} \times \mathbb{R}^3$.
\end{enumerate}
For all test-elements $\psi \in C^1_c([0, \, T[ \times \overline{\Omega}; \, \mathbb{R}^N)$ and $\eta \in C^1_c([0, \, T[ \times \Omega; \, \mathbb{R}^3)$, these fields are subject to the following integral relations
\begin{align}
& \label{INTRELmass} -\int_{Q_T} \rho \cdot (\partial_t \psi + (v \cdot \nabla) \psi) \, dxd\tau-\int_{Q_T} J \, : \, \nabla \psi \, dxd\tau=  \int_{\Omega} \rho^0(x) \cdot \psi(x, \, 0) \, dx \, ,\\
& \label{INTRELmoment} - \int_{Q_T} \varrho \, v \cdot (\partial_t \eta + (v \cdot \nabla) \eta) \, dxd\tau + \int_{Q_T} \mathbb{S}(\nabla v) \, : \, \nabla \eta \, dxd\tau =\nonumber\\
& \qquad \int_{\Omega} v^0(x) \cdot \eta(x, \, 0) \, dx + \int_{Q_T} \hat{p}(\rho) \, \divv \eta \, dxd\tau + \int_{Q_T} \rho\,  b \cdot \eta \, dxd\tau  \, .
\end{align}
In order to solve the Maxwell-Stefan equations \eqref{DIFFUSFLUX}, \eqref{maxsteftensor}, we must make sense of the generalised driving forces. In particular, we must explain by possibly vanishing mass densities why we are allowed to build the spatial gradient of $\mu := \hat{\mu}(\hat{p}(\rho), \, \hat{x}(\rho))$. To this aim we introduce for large parameters $k \in \mathbb{N}$ and vectors $\xi \in \mathbb{R}^N$ a cut-off operator $[\xi]^k := ([\xi_1]^k, \ldots, [\xi_N]^k)$ and for $s \in \mathbb{R}$, we define $[s]^k := \text{sign}(s) \, \min\{|s|, \, k\}$. For a vector $\rho$ subject to \eqref{massdensities}, we then require in addition to \eqref{massdensities}, \eqref{velocities}, \eqref{flux} the following additional technical condition:
\begin{enumerate}[(i)]
\addtocounter{enumi}{3}
 \item \label{Condimu1} The projection $\mathcal{ P}_{\{1^N\}^{\perp}} \hat{\mu}(\hat{p}(\rho), \, \hat{x}(\rho))$ satisfies $$[\mathcal{ P}_{\{1^N\}^{\perp}} \hat{\mu}(\hat{p}(\rho), \, \hat{x}(\rho))]^k \in W^{1,0}_2(Q_T; \, \mathbb{R}^N) \text{ for all } k \in \mathbb{N} \, . $$
\end{enumerate}
Obviously, the vector field $\hat{\mu}(\hat{p}(\rho), \, \hat{x}(\rho))$ is finite almost everywhere in the set 
\begin{align}\label{A+}
A^+ = \bigcap_{i = 1}^N A_i^+ \text{ with  } A_i^+ := \{(x, \,t) \in Q_T \, : \, \rho_i(x, \, t) > 0 \} \, .
\end{align}
Due to the condition \eqref{Condimu1}, the matrix $\nabla \mathcal{ P}_{\{1^N\}^{\perp}} \hat{\mu}(\hat{p}(\rho), \, \hat{x}(\rho)) $ is therefore identical with an element of $L^2(Q_T; \, \mathbb{R}^{N\times 3})$ at almost every point of $A^+$.

With $B(\rho)$ defined via \eqref{matrixB}, with $R := \text{diag}(\rho_1, \ldots, \rho_N)$ and $P = I - 1^N \otimes \hat{y}(\rho)$, solving the Maxwell-Stefan equations in the weak sense shall consist of two pointwise identities:
\begin{align}\label{maxstefweak1}
B(\rho(x, \, t)) \, J(x, \, t) = & - R \, P(x, \, t) \, (\nabla \mu(x, \,t) - b(x, \,t))\, \, \text{ for almost all } (x,t) \in A^+ \, ,\\[0.4cm]
\label{maxstefweak2} J^i(x, \, t) =&  0  \, \, \text{ for almost all } (x, \, t) \in Q_T \setminus A_i^+ \text{ such that } \hat{p}(\rho(x, \, t)) > 0 \, .
\end{align}
\begin{remark}
For a weak solution satisfying \eqref{maxstefweak1} and \eqref{maxstefweak2}, the flux $J$ remains non identified only on the pathological vacuum set. For positive total mass density -- or equivalently positive pressures -- $J$ satisfies the Maxwell-Stefan equations (cf. \eqref{maxstefweak1}) and exhibits the expected behaviour in the dilute limit (cf. \eqref{maxstefweak2}).
\end{remark}
\begin{theorem}\label{MAIN}
We assume that $\Omega \subset \mathbb{R}^3$ is bounded and of Lipschitz class. Let $T > 0$. For the data, we assume that
\begin{enumerate}[(a)]
\item The functions $g_i$ are subject to conditions {\rm (A)}, and the numbers $\alpha_1,\ldots,\alpha_N$ occurring in {\rm (A5)} satisfy $\max_{i= 1,\ldots,N} \alpha_i < \frac{3}{2}$;
\item The functions $\{f_{i,k}\}_{i<k}$ are subject to the conditions {\rm (B)};
\item $\rho^0 \in L^{\infty}(\Omega; \, \mathbb{R}^N_+)$ is such that $\underset{x \in \Omega}{\text{essinf}}\, \rho^0_i(x) > 0$ for $i=1,\ldots,N$,  and $v^0 \in L^2(\Omega; \, \mathbb{R}^3)$;
\item The external forcing $b$ belongs to $L^{\infty}(Q_T; \, \mathbb{R}^{N\times 3})$.
\end{enumerate}
Then, the problem $(P)$ possesses a global weak solution $(\rho, \, v, \, J)$. In particular, it is subject to \eqref{massdensities} with $\gamma = \max_{i= 1,\ldots,N} \alpha_i/(\max_{i= 1,\ldots,N} \alpha_i-1) > 3$ and satisfies \eqref{velocities}, \eqref{flux}, \eqref{Condimu1}. The equations \eqref{mass}, \eqref{momentum} are satisfied in the sense of the integral identities \eqref{INTRELmass}, \eqref{INTRELmoment}, and the Maxwell-Stefan equations \eqref{maxsteftensor} in the sense of the pointwise identities \eqref{maxstefweak1}, \eqref{maxstefweak2}.
\end{theorem}
\begin{remark}
In analogy to the results known for single components Navier-Stokes equations (\cite{lionsfils}, \cite{feinovpet}, a. o.) and for multicomponent systems (\cite{dredrugagu17a}), we expect that a similar result holds for $\max_{i= 1,\ldots,N} \alpha_i < 3$. To verify this claim would however considerably increase the technicality.
\end{remark}
{\bf Structure of the paper.} Due to several preliminary investigations, the steps of regularising the PDE system and constructing sufficiently regular approximate solution are in principle clear enough. We focus in this paper on the new aspects in the estimates and the convergence proof. In the next section, the general strategy is first exposed: Change of variables, new estimates, compactness argument. The remainder of the paper is the technical part: We present separately in two separate sections the detailed proofs concerning the change of variables and the resulting new robust estimates. The final section of the paper discusses the convergence of a typical approximation scheme.

\section{Main ideas}\label{MAINIDEAS}

As explained in the previous section, the first main question is how to make sense of the Maxwell--Stefan equations \eqref{DIFFUSFLUX} even for extreme states of the system, in particular dilute or rarefied ones. Both phenomena cannot be avoided for weak solutions.
From the technical viewpoint of analysis, we introduce a new parametrisation of the variable $\rho$ in the state space $\mathbb{R}^N_+$, in such a way as to obtain weak parabolic estimates for auxiliary variables $w_1, \ldots, w_{N}$ independently of vanishing or blow--up of the total mass density. We show that a refined analysis of the dissipation due to diffusion, when combined with the growth assumptions (B3) for the friction coefficients in the low-pressure regime, shall provide the robust parabolic estimates we need. The variables $w_1,\ldots,w_N$ shall have the meaning of some normalisation of the mass densities to reference pressure.

Thanks to these estimates, we can rephrase the pressure as a function of the total mass density and the parabolic variable $w$ via $p = \hat{p}(\rho) = \mathscr{P}(\varrho, \, w)$. The function $\mathscr{P}$ is increasing in the first component, and bounded differentiable in the remaining parabolic components. This structure allows to prove that the total mass density associated with a weak solution is in a compact set of $L^1(Q_T)$. We now proceed to explaining these main ideas in more details.

\subsection{A fundamental direction vector}\label{VECTORU}

Consider a solution $\rho, \, v$ to the problem $(P)$. In the absence of sources, external fluxes and forces, the energy (in)equality for the system reads as follows: For all $0 < t \leq T$
\begin{align}\label{EE}
&  \int_{\Omega} \{h(\rho(x, \, t)) + \frac{\varrho(x, \, t)}{2} \, |v(t, \, x)|^2\} \, dx + \int_{Q_t} \{\mathbb{S} \, : \,  \nabla v - J \, : \,  \nabla \mu\} \, dxd\tau\nonumber\\
& \qquad  \leq  \int_{\Omega} \{h(\rho^0(x)) + \frac{\varrho^0(x)}{2} \, |v^0(x)|^2\} \, dx \, .
\end{align}
As is well known -- this shall be recalled below in more details -- inversion of the Maxwell-Stefan equations \eqref{DIFFUSFLUX} provides a relationship $J^i = - \sum_{j=1}^N M_{i,j}(\rho) \, \nabla \mu_j$ with $M(\rho(x,t)) \in \mathbb{R}^{N\times N}$ symmetric, positive semi-definite, satisfying $\text{rk}(M(\rho(x,t))) = N-1$ and $\text{N}(M(\rho)) = \{1^N\}$. Hence \eqref{EE} provides, among others, a control on the quantity $\zeta := \sum_{i,j=1}^N M_{i,j}(\rho) \, \nabla \mu_i \cdot \nabla \mu_j$ in $L^1(Q_T)$, which represents the dissipation (or entropy production) of diffusion. 

To further characterise this $\zeta$, we use \eqref{CHEMPOT} and the chain rule compute the spatial gradient $\nabla \mu = D^2h(\rho) \, \nabla \rho$ with the Hessian $D^2h$ of the free energy. We obtain that
\begin{align*}
 \sum_{i,j=1}^N M_{i,j}(\rho) \, \nabla \mu_i \cdot \nabla \mu_j = \sum_{k,\ell=1}^N (D^2h(\rho) \, M(\rho) \, D^2h(\rho))_{k,\ell} \, \nabla \rho_k \cdot \nabla \rho_{\ell} \, .
\end{align*}
Due to the properties of $M(\rho)$, the matrix $D^2h(\rho) \, M(\rho) \, D^2h(\rho)$ is symmetric, positive semi--definite, and possesses rank $N-1$. Its kernel is the vector $u = u(\rho)$ solution to $D^2h(\rho) \, u = 1^N$. We next want to determine $u$. In order to compute the Hessian, we exploit the particular representation \eqref{FE}, or directly the formula \eqref{IDEALCHEMPOT}. If all $g_i(p)$ are strictly concave functions of class $C^2(]0, \, +\infty[)$ -- as required in assumptions (A) -- the implicit equation \eqref{PRESSURESTATE} for $\hat{p}$ can be used to compute that 
\begin{align}\label{pressderiv}
\partial_{\rho_j}\hat{p}(\rho) = -\frac{g_j^{\prime}(\hat{p}(\rho))}{\sum_{k=1}^N g_k^{\prime\prime}(\hat{p}(\rho)) \, \rho_k} \, .
\end{align}
Thus, setting $k_{B}\theta = 1$ for the sake of easier notation, we find for the Hessian of the free energy the expression
\begin{align}\label{HESS}
 D^2_{\rho_i,\rho_j}h(\rho) = \frac{1}{m_i\, m_j \, \hat{n}(\rho)} \, \left(\frac{\delta_{i,j}}{\hat{x}_i(\rho)} - 1\right) - \frac{1}{\sum_{k=1}^N g_k^{\prime\prime}(\hat{p}(\rho)) \, \rho_k}   \,  g_{i}^{\prime}(\hat{p}(\rho))   \, g_{j}^{\prime}(\hat{p}(\rho))  \, .
\end{align}
Recall in this point that $\hat{n}(\rho) = \sum_{j} (\rho_j/m_j)$ is the total number density. As will be shown below, this matrix is explicitly invertible (see the inversion Lemma \ref{inversehess}), so that we can now compute directly that the solution to $D^2h(\rho) \, u = 1^N$ is
\begin{align}\label{darstellunguprime}
 u_i(\rho) = & \rho_i \, \Big(m_i \, (1- \varrho \, g_i^{\prime}(\hat{p}(\rho))) + a(\rho)\Big) \, , \nonumber\\
 a(\rho) := & \sum_{j=1}^N g_j^{\prime}(\hat{p}(\rho)) \, (\varrho \,  g_j^{\prime}(\hat{p}(\rho))- 1) \, \rho_j\, m_j - \varrho \, \sum_{j=1}^N g_j^{\prime\prime}(\hat{p}(\rho)) \, \rho_j \, \, ,
\end{align}
in which $\varrho = \sum_j \rho_j$.

\subsection{Re-parametrisation of the state space}\label{PARAM}

We next introduce an parametrisation of the state space $\mathbb{R}^N_+$ using alternative variables. 
The aim is a description $\rho = X(s, \, w_1, \ldots, w_N)$, where $w$ corresponds to mass densities normalised at some fixed reference pressure $p_0$, and the parameter $s$ stands for $\hat{p}(\rho)$ or the pressure. For fixed $w$, the curve $s \mapsto X(s)$ shall be the characteristics following the vector $u$ of \eqref{darstellunguprime}.

To describe this procedure, we let $S_0 := \{\rho \in \mathbb{R}^N_+ \, : \, \sum_{i=1}^N g^{\prime}_i(p_0) \, \rho_i = 1\}$. We shall also abbreviate $\bar{v}^0_i := g^{\prime}_i(p_0)$, which is a fixed vector in $\mathbb{R}^N_+$. Due to the definition of the pressure state equation \eqref{PRESSURESTATE}, the identity $S_0 = \{ \rho \in \mathbb{R}^N_+ \, : \, \hat{p}(\rho) = p_0\}$ is valid.

For $w \in S_0$, consider on an interval $I = ]p_0 - \delta, \, p_0+ \delta[$ the system of ordinary differential equations
\begin{align}\label{ODE}
\dot{X}_i(s, \, w) = \frac{1}{X(s, \, w) \cdot 1^N} \, u_i(X(s; \, w)) \, \text{ for } s \in I , \quad X_i(p_0, \, w) = w_i \, . 
\end{align}
Here the vector field $u$ is given by \eqref{darstellunguprime}, and $\dot{X} := \frac{d}{ds}X$. Assume that $X$ is a solution to \eqref{ODE} such that $X$ assumes values in $\mathbb{R}^N_+$. Then we easily compute that
\begin{align}\label{identifypress}
 \frac{d}{ds} \hat{p}(X(s)) = & \nabla_{\rho}\hat{p}(X(s)) \cdot \dot{X}(s) = \frac{1}{X(s; \, w) \cdot 1^N} \, \nabla_{\rho}\hat{p}(X(s)) \cdot u(X(s; \, w))\nonumber\\
 = & \frac{1}{X(s; \, w) \cdot 1^N} \, D^2h(X(s))X(s)  \cdot u(X(s; \, w)) = 1 \, .
\end{align}
In order to justify the two latter equations, recall that \eqref{GIBBSDUHEMEULER} guaranties that $\nabla_{\rho}\hat{p}(\rho) = D^2h(\rho)\rho$, while the definition of $u$ implies that $D^2h(\rho) u(\rho) = 1^N$. Thus, \eqref{identifypress} implies that $\hat{p}(X(s)) = s + c$. For $s = p_0$, we use that $X(p_0, \, w) = w$. Since $w \in S_0$, we have $\hat{p}(X(s, \, w)) = \hat{p}(w) = p_0$ showing that $c = 0$ and that $\hat{p}(X(s, \, w)) = s$ for all $s > 0$ and all $w \in S_0$. Thus, the curves $s \mapsto X(s, \, w)$, $s \in I$ described by \eqref{ODE} are parametrised by the pressure. 

An important further property of the solution to \eqref{ODE} is the following. Consider any fixed vector $\eta \in \mathbb{R}^N$ such that $\eta \cdot 1^N = 0$ (short: $\eta \in \{1^N\}^{\perp}$). Using that $\dot{X}(s)$ is parallel to $u(X(s))$, that $D^2h(X(s)) u(X(s)) = 1^N$, we compute that $\frac{d}{ds} \eta \cdot \nabla_{\rho} h(X(s)) = D^2h(X(s))  \dot{X}(s) \cdot \eta = 0$. This means that 
\begin{align}\label{Proper}
\eta \cdot \nabla_{\rho} h(X(s, \, w)) = \eta \cdot \nabla_{\rho} h(X(p_0, \, w)) = \eta\cdot  \nabla_{\rho} h(w) \text{ for all } s\in I, \, w \in S_0\, , \eta \in \{1^N\}^{\perp} \, .
\end{align}
Anticipating on our results here below, we will show that the solution to \eqref{ODE} is global on the interval $I = ]0, \, +\infty[$ for all $w \in S_0$. Moreover, the map $X$ is bijective from $I \times S_0$ into $\mathbb{R}^N_+$.

Now, suppose that we apply to a solution vector $\rho(x, \, t)$ of the PDE system \eqref{mass} this change of variable, re-parametrising $\rho_i(x, \, t) = X_i(s(x, \,t), \, w(x, \, t))$ with $s(x, \, t) = \hat{p}(\rho(x, \, t))$ and a certain vector field $w: \, Q \rightarrow S_0$. Then for every $\eta \in \{1^N\}^{\perp}$ it follows by means of \eqref{Proper} that
\begin{align}\label{defmuw}
\eta \cdot \mu(x, \, t) = & \eta \cdot \nabla_{\rho} h(\rho(x, \,t)) = \eta \cdot \nabla_{\rho} h(X(s(x, \,t), \, w(x, \, t))) \nonumber\\
= & \eta \cdot \nabla_{\rho} h(w(x, \, t)) = \eta \cdot \hat{\mu}(p_0, \, \hat{x}(w(x, \, t))) =:  \eta \cdot \mu^{w}(x, \, t)    \, .
\end{align}
Recall at this point $\hat{\mu}$ is a map defined on $]0, \, + \infty[ \times S^1_+$ via $\hat{\mu}(p, \, x) =g(p) + \frac{1}{m} \, \ln x$ (cf. \eqref{IDEALCHEMPOT}).
Due to \eqref{defmuw}, the dissipation $\zeta := \sum_{i,j=1}^N M_{i,j}(\rho) \, \nabla \mu_i \cdot \nabla \mu_j$ due to diffusion possesses the equivalent representation $\zeta = \sum_{i,j=1}^N M_{i,j}(\rho) \,\nabla \mu^w_i \cdot \nabla \mu^w_j$.

\subsection{A robust estimate}\label{ROBUSTSEC}

For $(x, \, t) \in Q_T$, consider $\rho = \rho(x, \, t) \in \mathbb{R}^N_+$. Then $\rho$ is represented $\rho = X(s, \, w)$ with $s = s(x, \, t) = \hat{p}(\rho(x,\, t))$ and a unique normalisation $w = w(x, \, t) \in S_0$. Our idea -- to be proved in details in the subsequent sections -- is that the growth conditions (B3) for the friction coefficients of the Maxwell-Stefan system allow to prove for the matrix $M(\rho)$ an inequality from below:
\begin{align}\label{FUNDAM2}
 M(\rho) \geq & d_0\, P^T(\hat{y}(\rho)) \, W \, P(\hat{y}(\rho))\, ,\quad  W := \text{diag}(w_1,\ldots, w_N) \, , \, P(y) := I - 1^N \otimes y \, ,
\end{align}
where $d_0 > 0$ is a constant and $y = \hat{y}(\rho) = \rho/(\sum_{i} \rho_i)$ are the mass fractions.

This means that the control for of the integral of $ \zeta = M(\rho) \, \nabla \mu^w \cdot \nabla \mu^w$ next implies an estimate in in $L^1(Q_T)$ for
\begin{align}\label{encorezeta}
\zeta =  M(\rho) \, \nabla \mu^w \cdot \nabla \mu^w \geq & d_0 \, \sum_{i=1}^N w_i \, |\nabla \mu^w_i - \hat{y}(\rho) \, \nabla \mu^w|^2\nonumber\\
= & d_0 \, w\cdot 1^N \, \sum_{i=1}^N \hat{y}_i(w) \, |\nabla \mu^w_i - \hat{y}(\rho) \, \nabla \mu^w|^2 \, 
\end{align}
with the mass fractions $\hat{y}_i(w) := w_i/ \sum_j w_j$ associated with the normalised state $w_1,\ldots,w_N$. 
From the identity \eqref{encorezeta}, we now proceed in two steps. Using first the convexity of $|\cdot|^2$, and the fact that $\hat{y}(w)$ is a vector of fractions, we obtain that
\begin{align}\label{stepone}
 \zeta \geq d_0 \, w\cdot 1^N \, |\sum_{i=1}^N \hat{y}_i(w) \, \nabla \mu^w_i - \hat{y}(\rho) \, \nabla \mu^w|^2 \, .
\end{align}
Next, we expand in \eqref{encorezeta} via
\begin{align*}
\sum_{i=1}^N \hat{y}_i(w) & \, |\nabla \mu^w_i - \hat{y}(\rho) \, \nabla \mu^w|^2 =\sum_{i=1}^N \hat{y}_i(w) \, |\nabla \mu^w_i - \hat{y}(w) \, \nabla \mu^w + \hat{y}(w) \, \nabla \mu^w - \hat{y}(\rho) \, \nabla \mu^w|^2\\
\geq & \frac{1}{2} \,  \sum_{i=1}^N \hat{y}_i(w) \, |\nabla \mu^w_i - \hat{y}(w) \, \nabla \mu^w|^2 - \sum_{i=1}^N \hat{y}_i(w) \, |\hat{y}(w) \, \nabla \mu^w - \hat{y}(\rho) \, \nabla \mu^w|^2 \\
\geq & \frac{1}{2} \,  \sum_{i=1}^N \hat{y}_i(w) \, |\nabla \mu^w_i - \hat{y}(w) \, \nabla \mu^w|^2 - \frac{\zeta}{d_0 \, w\cdot 1^N} \, .
\end{align*}
Thus, combining with \eqref{encorezeta}, we get
\begin{align}\label{encorezeta2}
 \sum_{i=1}^N \hat{y}_i(w) \, |\nabla \mu^w_i - \hat{y}(w) \, \nabla \mu^w|^2 \leq \frac{4}{d_0 \, w\cdot 1^N} \,  \zeta \, .
\end{align}
Now, a brief calculation using the definitions of the mass fractions $\hat{y}(w)$ and the number fractions $\hat{x}(w)$ shows that $\hat{y}_i(w) = m_i \, \hat{x}_i(w) / \sum_j m_j \,  \hat{x}_j(w)$. Therefore, due to the fact that $\hat{x}$ attains values in $S^1_+$, we show that
\begin{align*}
\hat{y}(w) \,\nabla \mu^w = & \sum_{i=1}^N \hat{y}_i(w) \, \frac{1}{m_i} \, \nabla \ln \hat{x}_i(w) = \frac{1}{m \cdot \hat{x}(w)}  \, \sum_{i=1}^N \hat{x}_i(w)  \, \nabla \ln \hat{x}_i(w)  = 0 \, .
\end{align*}
Thus, the inequality \eqref{encorezeta2} yields
\begin{align}\label{ROBUST}
\zeta \geq & \frac{d_0}{4} \, w\cdot 1^N \, \sum_{i=1}^N \hat{y}_i(w) \, |\nabla \mu^w_i|^2 = \frac{d_0 \, w\cdot 1^N}{m \cdot \hat{x}(w)} \, \sum_{i=1}^N \frac{1}{m_i} \, |\nabla \sqrt{\hat{x}_i(w)}|^2 \, .
\end{align}
This also implies a bound for the gradient of $w$. Recall indeed that $w \cdot \bar{v}^0 = 1$, due to the fact that $w \in S_0$. This allows to show the identity $w_i = \hat{x}_i(w)/(m_i \, \sum_{j} \hat{x}_j(w) \, (\bar{v}^0_j/m_j))$. For $a_i := \sqrt{\hat{x}_i(w)/m_i}$, we thus get $w_i = a_i^2/(a^2 \cdot \bar{v}^0)$. We clearly obtain that $|\nabla w|^2 \leq c_1 \, |\nabla a|^2$, where $c_1$ depends on the numbers $\min m,\, \max m, \, \min \bar{v}^0, \, \max \bar{v}^0$. Since also $w\cdot 1^N \geq \min \bar{v}^0$ we finally obtain that
\begin{align}\label{ROBUST2}
\zeta \geq  d_0 \, c_2 \, |\nabla w|^2 \, .
\end{align}

\subsection{Proving compactness}\label{PQ}

Imagine that we have approximated the problem $(P)$ with a sequence of problems $(P_{\sigma})$ easier to solve. For example, we discuss below the case of replacing the Onsager operator with a stabilisation $M^{\sigma}(\rho) := M(\rho) + \sigma \, {\rm I}$ of full rank in order to obtain a parabolic problem. The question is now whether the vector of mass densities $\rho^{\sigma}$ solution to $(P_{\sigma})$ converges, at least for an appropriate subsequence, to a weak solution for $(P)$. In view of the nonlinearity of the system, it is clear that obtaining strong convergence is a necessary step to prove existence for the limit problem.

In a series of studies (in particular in \cite{dredrugagu17c}, \cite{druetmixtureincompweak}) devoted to this type of problems, we have shown that the compactness problem is reduced to obtaining the strong convergence of the hyperbolic components, that is, the sequence of total mass densities $\varrho^{\sigma}$. Indeed, once the latter property is known to hold, the weak parabolic estimates \eqref{ROBUST2} allow to show the compactness of the entire vector $\rho_1^{\sigma}, \ldots,\rho^{\sigma}_N$ by means of adapted Aubin-Lions (J.-L.) techniques.

Now, the compactness of the total mass densities is essentially obtained by adapting the method of Lions (P.-L.) for single component Navier-Stokes to our case. We show here in essence how to derive this adaptation from the estimate \eqref{ROBUST2}.

Assume first that the growth of the pressure $\rho \mapsto \hat{p}(\rho)$ is sufficiently strong. This means that there are positive $c_0, \, c_1$ such that $\hat{p}(\rho) \geq c_0 \, |\rho|^{\gamma} - c_1$ with $\gamma > 3/2$. To simplify the technical discussion, we will in fact assume for the proof of our main result that $\gamma > 3$, which corresponds to choosing $\max_{i = 1,\ldots,N} \alpha_i < 3/2$ in (A5). In this case we elementarily obtain from the mathematical theory of Navier-Stokes equations an \emph{a priori} estimate in $L^{1+1/\gamma}(Q_T)$ for the pressure, and in $L^{1+\gamma}(Q_T)$ for the density, so that the product $p \, \varrho$ is at least in $L^1(Q_T)$.

Then, due to many investigations (work of Lions, Feireisl and, for the multicomponent case, also our study \cite{dredrugagu17c} a.o.), the validity of the Lions argument for the compactness of the total mass density is reduced to showing
\begin{align}\label{cacs}
\lim_{\sigma \rightarrow 0} \int_{Q_t} p^{\sigma} \, \varrho^{\sigma} \, dxd\tau \geq \int_{Q_t} p \, \varrho \, dxd\tau \, ,
\end{align}
where $p$ is the weak limit of $p^{\sigma}$ in $L^{1+\frac{1}{\gamma}}(Q_T)$ and $\varrho$ the weak limit of $\varrho^{\sigma}$ in $L^{1+\gamma}(Q_T)$. To prove the property \eqref{cacs}, we essentially employ the monotonicity of the curves $s \mapsto X(s, \, w)$ of \eqref{ODE}. The strict monotonicity allows to introduce for $r >0$ the coordinate transformation $(s, \, w) \mapsto (r, \, w)$ via $r = \sum_{j=1}^N X_j(s, \, w)$, which defines an implicit function $s = \mathscr{P}(r, \, w)$. 
In particular, the inverse map $r \mapsto \mathscr{P}(r, \, w)$ is increasing on $]0, \, + \infty[$.

Thus, for some sequence $\{\sigma_n\}_{n \in \mathbb{N}}$, let us change variables and represent $p^{\sigma_n}(x, \, t) = s^n(x, \, t)$ and $\varrho^{\sigma_n}(x, \, t) = X(s^n(x,t), \, w^n(x,t)) \cdot 1^N$. Thanks to the procedure just described, we find $s_n = \mathscr{P}(\varrho^{\sigma_n}, \, w^n)$. Then, the Lions compactness argument is reduced to obtaining
\begin{align*}
 \lim_{n \rightarrow \infty} \int_{Q_t} p^{\sigma_n} \, \varrho^{\sigma_n} \, dxd\tau =  &\lim_{n \rightarrow \infty} \int_{Q_t} \mathscr{P}(\varrho^{\sigma_n}, \, w^n) \, \varrho^{\sigma_n} \, dxd\tau \geq \int_{Q_t} p \, \varrho \, dxd\tau \, ,
\end{align*}
This property has been proved among others in \cite{dredrugagu17c}. We recall it briefly, abbreviating for simplicity $\varrho^n = \varrho^{\sigma_n}$ etc.
We start from the inequality $(\mathscr{P}( \varrho^n, \, w^n) - \mathscr{P}(\tilde{\varrho}, \, w^n)) \, (\varrho^n - \tilde{\varrho}) \geq 0$ valid for arbitrary positive $\tilde{\varrho}$. Integrated against any smooth nonnegative testfunction $\phi$, this yields
\begin{align*}
 \int_{Q_t} \mathscr{P}( \varrho^n, \, w^n) \, \varrho^n \, \phi \, dxd\tau \geq & \int_{Q_t} \mathscr{P}( \varrho^n, \, w^n)\, \tilde{\varrho}\, \phi \, dxd\tau  + \int_{Q_t} \mathscr{P}(\tilde{\varrho}, \, w^n) \, (\varrho_n - \tilde{\varrho}) \, \phi \, dxd\tau \, .
\end{align*}
For $m \in \mathbb{N}$, we let $\tilde{\varrho}^m$ be a smooth positive function approximating $\varrho$ 
\begin{align}\label{approxprop}
 \tilde{\varrho}^m \rightarrow \varrho \text{ in } L^{1+\gamma}(Q_T)\, ,
\end{align}
and such that we can find for all $m \in \mathbb{N}$ some positive constants $a_m, \, b_m$ for which $0< a_m \leq \tilde{\varrho}^m \leq b_m < + \infty$. Suppose now that for $m$ fixed, we can establish for all $n$ the property
\begin{align}\label{PDERIV}
\sup_n |\diff_{w}\mathscr{P}(\tilde{\varrho}^m, \, w^n)| \leq C(a_m, \, b_m) < + \infty\, .
\end{align}
Then, \eqref{PDERIV} can be combined with the uniform estimate \eqref{ROBUST2} for $\nabla w^n$ in $L^2$. The chain rule for Sobolev functions yields a uniform bound in $L^2(Q_T; \, \mathbb{R}^3)$ for the sequence $\{\nabla \mathscr{P}(\tilde{\varrho}^m, \, w^n)\}$. Since $\partial_t\varrho_n$ is uniformly bounded in $L^2(0,T; \, (W^{1,6\gamma/(5\gamma-6)}(\Omega))^*)$ (see the uniform estimates in \cite{dredrugagu17b}), a div-curl type argument then shows that $$\int_{Q_t} \mathscr{P}(\tilde{\varrho}^m, \, w^n) \, (\varrho^n - \tilde{\varrho}^m) \, \phi \, dxd\tau \rightarrow \int_{Q_t} \beta_m \, (\varrho - \tilde{\varrho}_m) \, \phi \, dxd\tau \, .$$
Here $\beta_m$ is a weak limit of $\mathscr{P}(\tilde{\varrho}^m, \, w^n)$. We can show that the functions $\beta_m$ are bounded in $L^{1+1/\gamma}(Q_T)$ independently of $m$. Indeed, owing to the pressure growth, there are some constants $c_0,c_1$ such that $c_0 \, (r^{\gamma} - 1) \leq \mathscr{P}(r, \, w) \leq c_1 \,  (r^{\gamma} + 1)$, and \eqref{approxprop} is valid. 
Thus
\begin{align*}
 \liminf_{n \rightarrow \infty}  \int_{Q_t} \mathscr{P}( \varrho^n, \, w^n) \, \varrho^n \, \phi \, dxd\tau \geq & \int_{Q_t} p\, \tilde{\varrho}^m\, \phi \, dxd\tau + \int_{Q_t} \beta_m \, (\varrho - \tilde{\varrho}^m) \, \phi \, dxd\tau \, .
\end{align*}
Letting $m$ tend to $\infty$, we get \eqref{cacs}. Therefore we can expect that, for typical approximation schemes, the total mass densities converges strongly in $L^{1}(Q_T)$.


The remainder of the paper is devoted to providing the complete proofs of the ideas sketched in this section. At this level, a somewhat more technical presentation cannot be completely avoided.

\section{Analysis of the characteristic curves}

We recall the notation $\bar{v}^0 := g^{\prime}_i(p_0)$ with a reference pressure $p_0 \in ]0, \, + \infty[$.

%
%
%

We define the vector field $u = u(\rho)$ according to \eqref{darstellunguprime}. This means that $u$ solves $D^2h(\rho) \, u = 1^N$.
We consider for $w \in S_0 := \{\rho \in \mathbb{R}^N_+ \, : \, \sum_{i=1}^N \bar{v}^0_i \, \rho_i = 1\}$ the ordinary differential equations
\begin{align}\label{ODE2}
\dot{X}_i(s; \, w) = \frac{1}{X(s; \, w) \cdot 1^N} \, u_i(X(s; \, w)) \text{ for } s \in ]0, \, + \infty[, \quad X(p_0; \, w) = w \, . 
\end{align}
\begin{lemma}\label{characteristics}
Assume that the functions $g_i$ comply to the assumptions {\rm (A)}. Then, for all $w \in S_0$, the equations \eqref{ODE2} possess a global solution $s \mapsto X(s, \, w)$ of class $C^1(]0, \, + \infty[;\, \mathbb{R}^N_+)$. The map $w \mapsto X(s, \, w)$ is continuously differentiable on $S_0$ for all $s > 0$. For each $\rho \in \mathbb{R}^N_+$, the equations $\rho = X(s, \, w)$ possess a unique solution $(s, \, w) \in ]0, \, + \infty[ \times S_0$. Moreover, the following bounds are valid:
\begin{enumerate}[(a)]
\item \label{Xtot} $ \frac{1}{\max_{j=1,\ldots,N} g_j^{\prime}(s)} \leq |X(s, \, w)|_1 \leq \frac{1}{\min_{j=1,\ldots,N} g_j^{\prime}(s)}$ for all $(s, \, w) \in ]0, \, + \infty[ \times S_0$;
\item \label{Quotients} For $k = 1,\ldots,N$, the quotient $F_k(s, \, w) := X_k(s, \, w)/ w_k$ is bounded below (above) by a function $\underline{F}_k$ ($\bar{F}_k$) depending only on $s$. On compact sets of $]0, \, + \infty[$, the functions $\underline{F}_k, \, \bar{F}^k$ are continuous, bounded and positive;
\item \label{sderiv}  $|\partial_sX(s, \, w)| \leq  (\max m - \min m) \, (1 + \frac{ \max g^{\prime}(s)}{\min g^{\prime}(s)}) + \frac{\max |g^{\prime\prime}(s)|}{\min g^{\prime}(s))^2}$;
\item \label{wderiv} Denote $\diff_w$ the tangential gradient on $S_0$. Then there is $C>0$ such that $| \diff_{w}X(s, \, w)| \leq C \, \sup_{k} \bar{F}_k(s) \, (1+ \frac{\sup g^{\prime}(s)}{\inf g^{\prime}(s)})$ with $\bar{F}$ from \eqref{Quotients}.
\end{enumerate}
\end{lemma}
\begin{proof}
First we shall identify the solution to \eqref{ODE2} almost explicitly. If $X$ is a solution to \eqref{ODE2} assuming values in $\mathbb{R}^N_+$, then $s = \hat{p}(X(s, \, w))$ (see \eqref{identifypress} and the subsequent explanations). Due to \eqref{FE} and \eqref{CHEMPOT}, we have $\nabla_{\rho}h(\rho) = \hat{\mu}(\hat{p}(\rho), \, \hat{x}(\rho))$  for all $\rho \in \mathbb{R}^N_+$. If we combine these facts, then we see that
\begin{align}\label{Stern}
\partial_{\rho_i}h(X(s, \, w)) = & g_i(\hat{p}(X(s, \, w))) +\frac{1}{m_i} \, \ln \hat{x}_i(X(s, \, w)) \nonumber\\
 = & g_i(s) +\frac{1}{m_i} \, \ln \hat{x}_i(X(s, \, w)) \, .
 \end{align}
We also recall that $\frac{d}{ds} \nabla_{\rho} h(X(s, \, w)) \cdot \eta = 0$ for all $\eta \in \{1^N\}^{\perp}$ (see \eqref{Proper}). For $j = 1,\ldots,N$ and $k \in \{1,\ldots,N\}$ arbitrarily fixed, we multiply in \eqref{Stern} with the vector $e^j - e^k \in \{1^N\}^{\perp}$, where $e^1,\ldots,e^N$ are the standard basis vectors. We obtain that
\begin{align*}
&  g(p_0) \cdot (e^j -e^k) + \frac{1}{m_j} \, \ln \hat{x}_j(w) - \frac{1}{m_k} \, \ln \hat{x}_k(w) = \\
& \qquad = (e^j -e^k) \cdot \nabla_{\rho} h(w) = (e^j -e^k) \cdot \nabla_{\rho} h(X(s, \, w))\\
& \qquad =  g(s) \cdot (e^j -e^k) + \frac{1}{m_j} \, \ln \hat{x}_j(X(s, \, w)) - \frac{1}{m_k} \, \ln \hat{x}_k(X(s, \, w)) ) \, .
\end{align*}
For the fraction $\hat{x}_j(X(s, \, w))$, the latter yields a representation
 \begin{align}\label{hatchinull}
  \hat{x}_j(X(s, \, w)) = \hat{x}_j(w) \, \left(\frac{\hat{x}_k(X(s, \, w))}{\hat{x}_k(w)} \right)^{\frac{m_j}{m_k}} \, \exp(m_j \, (g(p_0)-g(s)) \cdot (e^j-e^k)) \, .
 \end{align}
We sum up over $j=1,\ldots,N$ and, since $\hat{x}$ maps into $S^1_+$, obtain that
\begin{align}\label{implicitfunction}
1 = \sum_{j=1}^N \hat{x}_j(w) \, \left(\frac{\hat{x}_k(X(s, \, w))}{\hat{x}_k(w)} \right)^{\frac{m_j}{m_k}} \, \exp(m_j \, (g(p_0)-g(s)) \cdot (e^j-e^k)) \, .
\end{align}
With $a_j = a_j(s, \, w) := \hat{x}_j(w) \, \exp(m_j \, (g_j(p_0)-g_j(s)))$, the identity \eqref{implicitfunction} is an equation of the form $\sum_{j=1}^N a_j \, \chi^{m_j} = 1$ for the variable $\chi := (\hat{x}_k(X(s, \, w))/\hat{x}_k(w))^{1/m_k} \, \exp(- (g_k(p_0)-g_k(s)))$. A brief analysis recalled in Lemma \ref{algebraic} shows that this equation possesses a uniquely determined and strictly positive root $\chi$. Moreover, we have $\chi = \hat{\chi}(a_1, \ldots, a_n)$ with a certain regular function of the entries of $a$. In particular, we can rely on the following estimates 
\begin{align}\label{ESTIMFORCHI}
\max_{i=1,\ldots,N} \chi^{m_i}(a) \leq \max\{1, \, (|a|_1)^{-\frac{\max m}{\min m}}\} \, , \, \quad \min_{i=1,\ldots,N} \chi^{m_i}(a) \geq \min\{1, \, (|a|_1)^{-\frac{\max m}{\min m}}\} \, .
\end{align}
Plugging this into \eqref{hatchinull}, we obtain the equation
\begin{align}\label{hatxindependentonk}
 \hat{x}_j(X(s, \, w)) = & a_j(s, \, w) \, \hat{\chi}^{m_j}(a_1, \ldots, a_n)  \text{ with } a_j(s, \, w) :=  \hat{x}_j(w) \, \exp(m_j \, (g_j(p_0)-g_j(s)))\, .
\end{align}
With the help of the latter relation, we can now derive a full representation of the solution to \eqref{ODE2}.
We make use of the identity $\sum_{i=1}^N g_i^{\prime}(s) \, X_i(s, \, w) = 1$ valid by definition of the pressure (cf. \eqref{PRESSURESTATE}), and further of $X_i(s, \, w) =\hat{n}(X(s, \, w)) \, m_i \, \hat{x}_i(X(s, \, w))$ (the definition of number fractions and number densities). For $\hat{n}$, we obtain the equivalent definition $\hat{n}(X(s, \, w)) = 1/\sum_{i=1}^N g_i^{\prime}(s) \, m_i \,  \hat{x}_i(X(s, \, w))$. Hence, invoking \eqref{hatxindependentonk} yields 
\begin{align}\label{hatn}
 \hat{n}(X(s, \, w)) = \frac{1}{\sum_{i=1}^N g_i^{\prime}(s) \, m_i \,  a_i(s, \, w) \, \hat{\chi}^{m_i}(a(s, \, w))} \, .
\end{align}
We combine \eqref{hatxindependentonk} and \eqref{hatn} with $X_k(s, \, w) = m_k \, \hat{n}(X(s, \, w)) \, \hat{x}_k(X(s, \, w))$ to obtain a representation
\begin{align}\label{representationX}
 X_k(s, \, w) = & \frac{ m_k \, a_k(s, \, w) \, \chi^{m_k}(a(s, \, w))}{\sum_{i=1}^N g_i^{\prime}(s) \, m_i \,  a_i(s, \, w) \, \hat{\chi}^{m_i}(a(s, \, w))} \\[0.4cm]
 a_k(s, \, w) := & \hat{x}_k(w) \, \exp(m_k \, (g_k(p_0)-g_k(s)))\, , \nonumber
\end{align}
and $\hat{\chi}(a) =: \chi$ defined to be the unique solution to the algebraic equation $\sum_{j=1}^N \chi^{m_j} \, a_j = 1$.

Now, we claim having found in \eqref{representationX} the global solution to \eqref{ODE2}, which can be verified briefly. Abbreviating $X = X(s, \, w)$, we indeed first notice that \eqref{representationX} implies, for all $j,k = 1,\ldots,N$, that
\begin{align*}
 \left(\frac{\hat{x}_j(X)}{\hat{x}_j(w)} \, \exp(m_j \, (g_j(p_0)-g_j(s)))\right)^{\frac{1}{m_j}} = \hat{\chi} =  \left(\frac{\hat{x}_k(X)}{\hat{x}_k(w)} \, \exp(m_k \, (g_k(p_0)-g_k(s)))\right)^{\frac{1}{m_k}} \, , 
\end{align*}
from which it is readily verified that $(e^j-e^k) \cdot (\hat{\mu}(s, \, \hat{x}(X)) - \hat{\mu}(p_0, \, \hat{x}(w))) = 0$. Since $k,j$ are arbitrary, this yields that the $X$ of \eqref{representationX} satisfies $\eta \cdot \nabla_{\rho}h(X(s,w)) = \eta \cdot \nabla_{\rho}h(w)$ for all $\eta \in \{1^N\}^{\perp}$. By these means, $\dot{X}(s, \, w)$ is orthogonal to $D^2h(X(s,w)) \, \eta$ for all $\eta \in \{1^N\}^{\perp}$, and we conclude that $\dot{X}(s, \, w)$ must be a vector parallel to $u(X(s, \, w))$. Next, the representation \eqref{representationX} also implies that $\sum_k g_k^{\prime}(s) \, X_k(s, \, w) = 1$ and therefore $s = \hat{p}(X(s, \, w))$ by the implicit definition of $\hat{p}$. Computing the derivative, this allows to show that $1 = \dot{X}(s, \, w) \, D^2h(X(s, \, w)) \, X(s, \, w)$. As we already know that $\dot{X}$ is parallel to $u = (D^2h)^{-1}1^N$, we find that $\dot{X}(s, \, w) = u(X(s, \, w))/(X(s, \, w) \cdot 1^N)$. Thus \eqref{representationX} provides the solution, and it can clearly be extended to a map defined on $]0, \, + \infty[ \times S_0$ with values in $\mathbb{R}^N_+$. 

Since $\hat{\chi}$ is differentiable, the formula \eqref{representationX} moreover shows that the map $w \mapsto X(s, \, w)$ is differentiable on $S_0$. 

We next prove the bijectivity of $X$, for which we first compute the inverse. The starting point is \eqref{hatchinull}, which we rephrase as
 \begin{align}\label{hatchinull1}
 \hat{x}_j(w) = \left(\frac{\hat{x}_k(w)}{\hat{x}_k(X(s, \, w))} \right)^{\frac{m_j}{m_k}} \, \exp(m_j \, (g(p_0)-g(s)) \cdot (e^k-e^j)) \, \, \hat{x}_j(X(s, \, w))  \, .
 \end{align}
We sum up, and we obtain the equation
\begin{align*}
 1 = \sum_{j=1}^N \left( (\frac{\hat{x}_k(w)}{\hat{x}_k(X(s, \, w))})^{\frac{1}{m_k}} \, e^{g_k(p_0)-g_k(s)} \right)^{m_j} \, \exp(- m_j \, (g_j(p_0)-g_j(s)) ) \, \hat{x}_j(X(s, \, w)) \, .
\end{align*}
This is an equation of the form $\sum_{j=1}^N b_j \, \phi^{m_j} = 1$.
It follows that $\phi = \hat{\chi}(b_1, \ldots, b_N)$ with $b_j = b_j(X)$ given as $b_j(X) := \hat{x}_j(X) \, \exp(- m_j \, (g_j(p_0)-g_j(\hat{p}(X))))$ and that
\begin{align*}
 \hat{x}_k(w) = \hat{x}_k(X) \, \phi^{m_k} \, \exp(-m_k \, (g_k(p_0)-g_k(\hat{p}(X)))) \, .
\end{align*}
In order to obtain $w$ as a function of $X$, we recall that $w \in S_0$ so that $w \cdot \bar{v}^0 = 1$. Therefore, we can represent $w_k$ as a function of $\hat{x}(w)$ via $w_k = m_k \, \hat{x}_k(w)/ (\sum_{\ell = 1}^N \bar{v}^0_{\ell} \, m_{\ell} \, \hat{x}_{\ell}(w))$. Thus
\begin{align}\label{representationw}
 w_k = & \frac{ m_k \, b_k(X) \, \hat{\chi}^{m_k}(b(X))}{\sum_{\ell=1}^N \bar{v}^0_{\ell} \, m_{\ell} \,  b_{\ell}(X) \, \hat{\chi}^{m_{\ell}}(b(X))} \text{ with }  b_k(X) :=  \hat{x}_k(X) \, \exp(-m_k \, (g_k(p_0)-g_k(\hat{p}(X))))\, ,
\end{align}
and $\hat{\chi}(b) =: \phi$ defined to be the unique solution to the algebraic equation $\sum_{j=1}^N \phi^{m_j} \, b_j = 1$.

For $\rho \in \mathbb{R}^N_+$, one therefore has $\rho = X(s, \, w)$ if and only if $s = \hat{p}(\rho)$ and $w = w(X)$ is given by \eqref{representationw} at $X = \rho$. This proves that $X$ is bijective.

Next we prove the estimates. 

Ad \eqref{Xtot}. From \eqref{representationX}, we directly deduce that the bound \eqref{Xtot} is valid.

Ad \eqref{Quotients}. By means of \eqref{representationX}, we readily identify $F_k(s, \, w) = X_k(s, \, w)/ w_k$ as
\begin{align}\label{Ffirst}
  F_k(s, \, w) = & \frac{\exp(m_k \, (g_k(p_0)-g_k(s))) \, \chi^{m_k}(a(s, \, w))}{(\sum_{j=1}^N w_j/m_j) \, (\sum_{j=1}^N g_j^{\prime}(s) \, m_j \,  a_j(s, \, w) \, \hat{\chi}^{m_j}(a(s, \, w)))} \, . 
\end{align}
Owing to $\sum_{j=1}^N (w_j/m_j) \geq w \cdot 1^N/|m|_{\infty} \geq 1/(|m|_{\infty} \, |\bar{v}^0|_{\infty})$, and to the definition of $\chi$, it follows that
\begin{align*}
  F_k(s, \, w) \leq \frac{|m|_{\infty} \, |\bar{v}^0|_{\infty}}{\min m \, \min g^{\prime}(s)} \, \exp(m_k \, (g_k(p_0)-g_k(s))) \, \chi^{m_k}(a(s, \, w))\, ,
  \end{align*}
  and now \eqref{ESTIMFORCHI} and the fact that $|a|_1 \geq \min_{\ell} \exp(m_\ell \, (g_\ell(p_0)-g_\ell(s)))$ yield
  \begin{align}\label{barF}
  F_k(s, \, w)  \leq & \frac{|m|_{\infty} \, |\bar{v}^0|_{\infty}}{\min m \, \min g^{\prime}(s)} \, \exp(m_k \, (g_k(p_0)-g_k(s))) \times \nonumber\\
  & \times \max\{1, \, \frac{1}{\min_{\ell} \exp(m_\ell \, (g_\ell(p_0)-g_\ell(s)))}\}^{\frac{\max m}{\min m}}  =:  \bar{F}_k(s) \, .
\end{align}
Similarly, we obtain the lower bound
\begin{align}\label{underlineF}
 F_k(s, \, w) \geq & \frac{\min m \, \min \bar{v}^0}{ |m|_{\infty} \, \max g^{\prime}(s)} \, \exp(m_k \, (g_k(p_0) - g_{k}(s))) \times \nonumber\\
 & \times \min\{1, \, \frac{1}{\sum_{j=1}^N \exp(m_j \, (g_j(p_0) - g_{j}(s)))}\}^{\frac{ \max m}{\min m}} =:  \underline{F}_k(s) \, .
\end{align}

We next estimate the derivatives.

Ad \eqref{sderiv}. Recall \eqref{darstellunguprime} to see that
\begin{align*}
 \frac{u_i(\rho)}{\rho_i} = & m_i -\sum_{j=1}^N g_j^{\prime}(\hat{p}(\rho)) \, \rho_j \, m_j + \varrho \, \big( -  m_i \, g_i^{\prime}(\hat{p}(\rho)) + \sum_{j=1}^N [m_j \, (g_j^{\prime}(\hat{p}(\rho)))^2 - g_j^{\prime\prime}(\hat{p}(\rho)) ] \, \rho_j \big) \, .
\end{align*}
Since $\sum_{j=1}^N g_j^{\prime}(\hat{p}(\rho)) \, \rho_j = 1$, we obtain the bound
\begin{align*}
\frac{|u_i(\rho)|}{\rho_i \, \varrho} \leq & \frac{\max m - \min m}{\varrho} + \max \{m \, g^{\prime}(\hat{p}(\rho))\} -   \min \{m \, g^{\prime}(\hat{p}(\rho))\} + \varrho \, |g^{\prime\prime}(\hat{p}(\rho))|_{\infty} \, .
\end{align*}
It follows that
\begin{align*}
\frac{|u(\rho)|}{\varrho} \leq & (\max m - \min m) \, (1 + \varrho \, \max g^{\prime}(\hat{p}(\rho)) ) + \varrho^2 \, |g^{\prime\prime}(\hat{p}(\rho))|_{\infty} \, .
\end{align*}
With the help of \eqref{Xtot}, we therefore verify for $X$ satisfying \eqref{ODE2} that
\begin{align*}
 |\dot{X}(s)| \leq \frac{|u(X(s))|}{X(s) \cdot 1^N} \leq (\max m - \min m) \, \left(1 + \frac{ \max g^{\prime}(s)}{\min g^{\prime}(s)}\right) + \frac{|g^{\prime\prime}(s)|_{\infty}}{\min g^{\prime}(s))^2} \, \, . 
\end{align*}

Ad \eqref{wderiv}. Due to the formula \eqref{representationX}, the map $w \mapsto X(s, \, w)$ is differentiable on $S_0$. To more easily compute the derivatives, we start again from the identity $(\nabla_{\rho} h(X(s, \, w)) - \nabla_{\rho} h(w)) \cdot \eta = 0$ for all $\eta \in \{1^N\}^{\perp}$ (see \eqref{Proper}). Consider $\tau^i := e^i - \nu_i \, \nu$ a tangent on $S_0$, and differentiate in the direction of $\tau^i$ the latter identity. Then, with $\diff_{w_i} := \tau^i \cdot \nabla_w$ denoting the tangential derivatives on $S_0$, we obtain for all $\eta \in \{1^N\}^{\perp}$ that $\eta \cdot( D^2h(X(s, \, w)) \, \diff_{w_i}X(s, \, w) -D^2h(w) \tau^i) = 0$. Therefore, there must exist some real number $r_i$ such that
\begin{align*}
D^2h(X(s, \, w)) \, \diff_{w_i}X(s, \, w) = D^2h(w) \tau^i + r_i \, 1^N \, .
\end{align*}
We multiply from left with $X(s, \, w)$. Recall that $D^2h(\rho) \, \rho = \nabla_{\rho}\hat{p}(\rho)$. Since $$\nabla_{\rho}\hat{p}(X(s,w)) \cdot \diff_{w_i}X(s, \, w) = \diff_{w_i} \hat{p}(X(s, w)) = \diff_{w_i} s = 0 \, ,$$ we identify $r_i$ as $- X(s,w) \cdot D^2h(w) \tau^i /X(s,w) \cdot 1^N$. Hence
\begin{align}\label{poule}
D^2h(X(s, \, w)) \, \diff_{w_i}X(s, \, w) = D^2h(w) \tau^i - \frac{D^2h(w) \tau^i \cdot X(s,w)}{X(s,w) \cdot 1^N} \, 1^N =: q^i(s,w) \, .
\end{align}
The inversion formula for $D^2h$ in Lemma \ref{inversehess}, and use of $\diff_{w_i}X(s, \, w) \cdot \hat{p}_{\rho}(X(s,w)) = 0$ yield the result
\begin{align}\label{Dw}
 & \diff_{w_i}X_j(s, \, w) \nonumber\\
 & = X_j(s,w) \, (m_j  \, q^i_j(s,w) - \frac{1}{\hat{p}_{\rho}(X(s,w))\cdot X(s,w)} \, \sum_{\ell=1}^N \hat{p}_{\rho_{\ell}}(X(s,w)) \, m_{\ell} \, X_{\ell}(s,w) \, q^i_{\ell}(s,w)) \nonumber\\
&  =X_j(s,w) \, (m_j  \, q^i_j(s,w) - \sum_{\ell=1}^N g_{\ell}^{\prime}(s) \, m_{\ell} \, X_{\ell}(s,w) \, q^i_{\ell}(s,w)) \, .
\end{align}
Here we also made use of the formula \eqref{pressderiv} for the derivative of the pressure.

In order to obtain estimates, we first use \eqref{HESS}. Together with the fact that $w \in S_0$, it implies that
\begin{align*}
D^2_{k,\ell} h(w) = \frac{1}{m_k} \, \left(\frac{\delta_{k,\ell}}{w_{\ell}} - \frac{1}{m_{\ell} \, \hat{n}(w)}\right) - \frac{g_{k}^{\prime}(p_0) \,   
g_{\ell}^{\prime}(p_0)}{\sum_{j} g_j^{\prime\prime}(p_0) \, w_j} \, .
\end{align*}
Therefore, for any $k,\ell$, the expressions $|D^2_{k,\ell} h(w) \, X_{\ell}(s, \, w)|$ are bounded by some uniform constant $C$ times the quotient $F_k(s, \, w) := X_{k}(s,w)/w_k$. For the functions $q$ introduced in \eqref{poule}, we obtain that
\begin{align*}
 |X_j(s,w) \, q^i_j(s,w)| \leq \sum_{\ell} |X_j(s,w) \, D^2_{j,\ell} h(w) \, \tau^{i}_{\ell}| + |D^2h(w) \tau^i \cdot X(s,w)| \leq C \, \sup_{k} F_k(s,\, w) \, .
 \end{align*}
Now, by means of \eqref{Dw}, we find that
\begin{align*}
| \diff_{w_i}X_j(s, \, w)| \leq |m|_{\infty} \, C \, \sup_{k} F_k(s,\, w) \, (1 + X_j(s,w) \, \sum_{\ell=1}^N g_{\ell}^{\prime}(s)) \, ,
\end{align*}
and, recalling \eqref{Xtot}, it follows that $| \diff_{w}X(s, \, w)| \leq \tilde{C} \, \sup_{k} F_k(s,\, w) \, (1+ \max g^{\prime}(s)/\min g^{\prime}(s))$. We invoke \eqref{barF} to see that $ | \diff_{w}X(s, \, w)|$ is bounded by a continuous function depending only on $s$. 
\end{proof}
We next provide the two auxiliary lemmas used in the preceding proof.
\begin{lemma}\label{algebraic}
For $a \in \mathbb{R}^N_+$, the equation $\sum_{j=1}^N a_j \, \chi^{m_j} = 1$ possesses a unique positive solution $\chi$ subject to the bounds $\min\{1, \, \frac{1}{|a|_1}\}^{\frac{1}{\min m}} \leq \chi \leq \max\{1, \, \frac{1}{|a|_1}\}^{\frac{1}{\min m}}$. The map $a \mapsto \chi =: \hat{\chi}(a)$ is continuously differentiable on $\mathbb{R}^N_+$. 
\end{lemma}
\begin{proof}
The function $F(\chi; \, a) = \sum_{j=1}^N a_j \, \chi^{m_j} - 1$ is monotone on $]0, \, +\infty[$, negative in $0$ and it tends to $+\infty$ for $\chi \rightarrow +\infty$. Clearly, there is exactly on positive $\hat{\chi}$ such that $F(\hat{\chi}; \, a) = 0$. Thus $\partial_{\chi} F(\hat{\chi}; \, a) > 0$ and the map $a \mapsto \hat{\chi}$ is differentiable. 

If $\hat{\chi} > 1$, then by definition $1 \geq |a|_1 \, \hat{\chi}^{\min m}$ proving the upper bound. 
If $\hat{\chi} < 1$, then $1 \leq |a|_1 \, \hat{\chi}^{\min m}$ implies the claimed lower bound. 
\end{proof}
\begin{lemma}\label{inversehess}
Let $\rho \in \mathbb{R}^N_+$, and $D^2h(\rho)$ is given by \eqref{HESS}. Then with $\Lambda := \sum_{k = 1}^N V_{\rho_k}^2 \, \rho_k \, m_k - V_p$ 
\begin{align*}
 (D^2h(\rho))^{-1}_{i,j} = m_i \, \rho_i \, \delta_{i,j} - (m_i \, \rho_i\, V_{\rho_i} \, \rho_j +  m_j \, \rho_j\, V_{\rho_j} \, \rho_i) + \Lambda(\rho) \, \rho_{i} \, \rho_j \, . 
\end{align*}
In these formula, $V(p,\, \rho) := \sum_i g_i^{\prime}(p) \, \rho_i$ and its derivatives $V_{\rho}, \, V_p$ are evaluated at $(\hat{p}(\rho), \, \rho)$.
\end{lemma}
\begin{proof}
Assume that $v, \, w \in \mathbb{R}^N$ are two vectors subject to $D^2h(\rho) \, v = w$. We multiply from left with $\rho$. 
Then, using that $\sum_{i} g_i^{\prime}(\hat{p}(\rho)) \, \rho_i = 1$, we find that $D^2h(\rho) \,\rho = - g^{\prime}(\hat{p}(\rho))/( \sum_j g^{\prime\prime}(\hat{p}(\rho)) \rho_j) = - V_{\rho}/V_p$, so that $-(V_{\rho}/V_p) \cdot v = D^2h(\rho) \, v \cdot \rho = w\cdot \rho$. Here $V_{\rho} = g^{\prime}(\hat{p}(\rho))$ and $V_p = \sum_j g^{\prime\prime}(\hat{p}(\rho)) \rho_j$. We introduce the positively homogeneous function 
\begin{align}\label{entropicpart}
k(\rho) =  \hat{n}(\rho) \, \hat{x}(\rho) \cdot \ln \hat{x}(\rho) \, , 
\end{align}
With the volume $V$ of \eqref{PRESSURESTATE}, we alternatively have 
\begin{align}\label{HESSCOMP}
D^2h(\rho) = D^2k(\rho) - \frac{1}{ \partial_{p}V(\hat{p}(\rho), \, \rho)} \,  \nabla_{\rho} V (\hat{p}(\rho), \, \rho) \otimes \nabla_{\rho} V (\hat{p}(\rho), \, \rho)\, .
\end{align}
Thus, $v$ is also a solution to $D^2k(\rho) \, v = w - (w\cdot\rho) \, V_{\rho}$, which means that
\begin{align*}
\frac{v_i}{m_i \, \rho_i} - \frac{1}{m_i \, n} \, \sum_{j=1}^N \frac{v_j}{m_j} = w_i - (w\cdot\rho) \, V_{\rho_i} \text{ for } i = 1,\ldots,N \, .
\end{align*}
We multiply with $m_i \, \rho_i \, V_{\rho_i}$ and sum up over $i=1,\ldots,N$ to find that
\begin{align*}
V_{\rho} \cdot v - \frac{1}{n} \, \sum_{j=1}^N \frac{v_j}{m_j} = \sum_{i=1}^N m_i \, \rho_i \, V_{\rho_i} \, (w_i - (w\cdot\rho) \, V_{\rho_i}) \, .
\end{align*}
Since we know that $V_{\rho} \cdot v = - V_p \, w\cdot \rho$, we can eliminate the quantity $\frac{1}{n} \, \sum_{j=1}^N \frac{v_j}{m_j}$. Hence
\begin{align*}
\frac{v_i}{m_i \, \rho_i} = &  w_i - (w\cdot\rho) \, V_{\rho_i} - \frac{1}{m_i} \, [V_p \, w\cdot \rho + \sum_{j=1}^N  m_j \, \rho_j \, V_{\rho_j} \, (w_j- (w\cdot\rho) \, V_{\rho_j})] \, .
\end{align*}
The claim follows. 
\end{proof}
In order to complete the picture about the change of variables, it remains to introduce the equivalent representation of the pressure needed to prove the compactness (see the Section \ref{PQ}). In particular, we must verify \eqref{PDERIV}.
\begin{lemma}\label{Lemmapress}
There is a function $\mathscr{P} \in C^1(]0, \, + \infty[ \times S_0)$ such that for $\rho \in \mathbb{R}^N_+$, $s > 0$ and $w \in S_0$ connected via $\rho = X(s, \, w)$, the identities $s = \hat{p}(\rho) = \mathscr{P}(\sum_{i} \rho_i, \, w )$ are valid. The function $\varrho \mapsto \mathscr{P}(\varrho, \, w)$ is strictly increasing. The norm of the tangential derivatives $\diff_w \mathscr{P}(\varrho, \, w)$ is bounded above by a continuous function of $\varrho$ only.
\end{lemma}
\begin{proof}
Suppose that $\rho = X(s, \, w)$ with $s = \hat{p}(\rho)$ and $w \in S_0$ obeying \eqref{representationw}. Then $\varrho := \sum_{i=1}^N \rho_i = X(s, \, w) \cdot 1^N$, which we regard as an implicit equation for $s$. Use of \eqref{ODE2} yields
\begin{align}\label{zweistern}
\frac{d}{ds}X(s, \, w) \cdot 1^N = \frac{u(X(s, \, w)) \cdot 1^N}{X(s, \, w) \cdot 1^N} \, . 
\end{align}
Since by definition $D^2h(X) \, u(X) = 1^N$, we see that $u(X) \cdot 1^N = D^2h(X) \, u(X) \cdot u(X) > 0$ (the Hessian of $h$ is positive definite). In fact, due to \eqref{HESSCOMP}
\begin{align*}
D^2h(X) \, u(X) \cdot u(X) \geq \frac{(V_{\rho}(s, \, X) \cdot u(X))^2}{|V_p(s, \, X)|} \, ,
\end{align*}
and now the identities $V_{\rho}/|V_p| = \hat{p}_{\rho}(X) = D^2h(X) \, X$ imply that
\begin{align*}
u(X) \cdot 1^N = & D^2h(X) \, u(X) \cdot u(X) \\
\geq & |V_p(s, \, X)| \, (D^2h(X) \, X \cdot u(X)^2\\
= & |V_p(s, \, X)| \, (X \cdot 1^N)^2 \, .
\end{align*}
With the help of \eqref{zweistern}, we attain the estimate $\dot{X}(s, \, w) \cdot 1^N \geq |V_p(s, \, X(s,w))| \, X(s, \, w) \cdot 1^N$. 
We can next bound $|V_p| = \sum_{j=1}^N |g_j^{\prime\prime}(s)| \, X_{j}(s, \, w)$ from below by $\min g^{\prime\prime}(s) \, X(s, \, w) \cdot 1^N$. Thus, Lemma \ref{characteristics}, \eqref{Xtot} yields
\begin{align}
 \label{choule2}  \frac{d}{ds}X(s, \, w) \cdot 1^N \geq \frac{\min g^{\prime\prime}(s)}{(\max g^{\prime}(s))^2} \, .
\end{align}
We define $\mathscr{P}(\varrho, \, w)$ to be the implicit solution to $\varrho = X(\mathscr{P}(\varrho, \, w), \, w) \cdot 1^N$. Then, the derivatives obey
\begin{align*}
\partial_{\varrho}\mathscr{P}(\varrho, \, w) =&  \frac{1}{X_s(\mathscr{P}(\varrho, \, w), \, w) \cdot 1^N}, \quad \diff_{w_i}\mathscr{P}(\varrho, \, w) =  - \frac{\diff_{w_i}X(\mathscr{P}(\varrho, \, w), \, w) \cdot 1^N }{X_s(\mathscr{P}(\varrho, \, w), \, w) \cdot 1^N} \, .
\end{align*}
We invoke \eqref{choule2} to see that $\partial_{\varrho}\mathscr{P} > 0$. Due to \eqref{choule2} and Lemma \ref{characteristics}, \eqref{wderiv}, we first obtain that $|\diff_{w}\mathscr{P}(\varrho, \, w)|$ is bounded above by a function of $s = \hat{p}(\rho) = \mathscr{P}$. Then, we invoke Lemma \ref{characteristics}, relation \eqref{Xtot}, and the fact that the monotonously decreasing functions $g^{\prime}$ are invertible, deducing that
\begin{align*}
 \min_{i=1, \ldots,N} [g^{\prime}_i]^{-1}(\frac{1}{X(s,w)\cdot 1^N}) \leq s \leq \max_{i=1, \ldots,N} [g^{\prime}_i]^{-1}(\frac{1}{X(s,w)\cdot 1^N}) \, .
\end{align*}
Thus, $|\diff_{w}\mathscr{P}(\varrho, \, w)|$ is bounded above by a continuous function of $X \cdot 1^N = \varrho$.
\end{proof}

\section{Analysis of the Maxwell-Stefan equations and the dissipation bound}

For this section we assume that the friction coefficients obey the assumptions (B). We consider the Maxwell-Stefan algebraic system 
\begin{align}\label{DIFFUSFLUX2}
(\sum_{k \neq i} f_{i,k}\, y_k) \, J^i - y_i \, \sum_{k \neq i} f_{i,k} \, J^k = - \rho_i \, \sum_{k = 1}^N  \, (\delta_{i,k} - y_k) \, (\nabla \mu_k - b^k) \, \text{ for } i =1,\ldots,N \, .
\end{align}
We define a matrix $B(\rho) = \{b_{i,k}(\rho)\}_{i,k=1,\ldots,N}$ (cp. \eqref{matrixB}) via
\begin{align}\label{matrixB2}
b_{i,k}(\rho) := \begin{cases}
               - f_{i,k} \, y_i & \text{ for } k \neq i \, ,\\
               \sum_{j \neq i} f_{i,j}\, y_j & \text{ for } k = i \, ,
              \end{cases}
\end{align}
so as to reformulate \eqref{DIFFUSFLUX2} as 
\begin{align}\label{Flux3}
B(\rho) \, J = - R \, P(y) \, (\nabla \mu-b) =: - d \, .
\end{align}
Here we recall that $R := \text{diag}(\rho_1, \ldots, \rho_N)$ and that $P(y) := I - 1^N \otimes y$. In all these equations we have used the natural transformation  $y = \hat{y}(\rho)$ for the mass fractions. In \eqref{DIFFUSFLUX2} and \eqref{matrixB2}, we have by assumption $f_{i,k} = f_{i,k}(\hat{p}(\rho), \, \hat{x}(\rho))$. 

It is well known and it can be easily checked that the matrix $B(\rho)$ is a singular \emph{M}-matrix with left kernel $\{1^N\}$ and right kernel $\{y\}$ (see among others \cite{giovan}, par. 7.7, \cite{bothedruetMS}). The use of generalised inverse to solve the Maxwell-Stefan equations is an original product by the first reference. Nevertheless we quote for convenience from the latter paper, recalling an invertibility result for \eqref{Flux3}.
\begin{lemma}\label{maxstefinverse}
For $\rho \in \mathbb{R}^N_+$, we denote $(B(\rho))^D$ the {\rm Drazin inverse} of the matrix $B(\rho)$ of \eqref{matrixB2}.
Then the unique solution $J \in \{1^N\}^{\perp}$ to the equations \eqref{Flux3} is given as $J := - (B(\rho))^D \, R \, d$.   
\end{lemma}
We now prove the second essential result for this paper in order to obtain the robust estimates, and introduce first the fundamental function to formulate our asymptotic conditions on the friction coefficients in the Maxwell-Stefan system. For $s \in ]0, \, + \infty[$, $w \in S_0$ we recall the definitions of quotients $F_i(s, \, w) = X_i(s, \, w)/w_i$ (cf. Lemma \ref{characteristics} and \eqref{Ffirst}). Define $\rho := X(s, \, w)$. Making use of the identity \eqref{representationw} for $w$, 
we obtain for $F_i$ an equivalent form as function of the pressure $s = \hat{p}(\rho)$ and the number densities $x = \hat{x}(\rho)$ via
\begin{align}\label{Fnew}
F_i(s, \, w) = \hat{F}_i(s, \, x) := \frac{\sum_{\ell = 1}^N \bar{v}^0_{\ell} \, m_{\ell} \, x_{\ell} \,  \exp(-m_\ell \, (g_{\ell}(p_0)-g_{\ell}(s))) \, \hat{\chi}^{m_{\ell}}(s, \, x)}{(\sum_{\ell = 1}^N g_{\ell}^{\prime}(s) \, m_{\ell} \, x_{\ell} ) \, \exp(-m_i \, (g_i(p_0)-g_i(s))) \,\hat{\chi}^{m_{i}}(s, \, x) }
\end{align}
with $\hat{\chi}(s, \, x)$ defined via solution of $\sum_{j=1}^N \hat{\chi}^{m_j} \, x_j \, \exp(-m_j \, (g_j(p_0)-g_j(s))) = 1$.
For $\rho = X(s, \, w)$ it follows that
\begin{align}\label{meaningF}
 \hat{F}_i(\hat{p}(\rho), \, \hat{x}(\rho)) = \rho_i/w_i = F_i(s, \, w) \, .
\end{align}
Making use of \eqref{Fnew}, \eqref{meaningF} and \eqref{frefbitte}, we can verify that the singular functions $\sigma_{i,k}(s, \, x)$ occurring in the condition (B3) are expressed via
\begin{align}\label{asympfric}
\sigma_{i,k}(s, \, x) := \frac{\hat{F}_i(s, \,x) \, \hat{F}_k(s, \, x)}{\sum_{\ell = 1}^N \hat{F}_{\ell}(s, \, x) \, m_{\ell} \, x_{\ell}} \, \sum_{\ell = 1}^N m_{\ell} \, x_{\ell}  \, .
\end{align}
Making use of \eqref{meaningF}, we have at $\rho = X(s, \, w)$ the equivalent expression
\begin{align}\label{meaningsmallg}
\sigma_{i,k}(s, \, \hat{x}(X(s, \, w))) = \frac{F_i(s, \, w) \, F_k(s, \,w)}{\sum_{\ell = 1}^N F_{\ell}(s, \, w) \, \hat{y}_{\ell}(X(s, \, w)) } \, .
\end{align}
We now prove the robust estimates.
\begin{lemma}\label{ROBUSTLEM}
 Assume that the functions $g_1, \ldots, g_N: ]0, \, + \infty[ \rightarrow \mathbb{R}$ satisfy the assumptions {\rm (A)}, and that the friction coefficients $f_{i,k}: \, ]0, \, + \infty[ \times S^1_+ \rightarrow \mathbb{R}_+$ ($1\leq i < k \leq N$) obey {\rm (B)} with some $p_1 > 0$.
Then, for $\rho \in \mathbb{R}^N_+$ represented as $\rho = X(s, \, w)$ with $s = \hat{p}(\rho) \leq p_1$ and $w \in S_0$, the following two claims are valid:
\begin{enumerate}[(1)]
 \item \label{LEMRES1} The inverse Maxwell-Stefan matrix satisfies $(B(\rho))^D \, R \geq \frac{1}{f_1} \, P^{\sf T}(\hat{y}(\rho)) \, W \, P^{\sf T}(\hat{y}(\rho))$, where $W = \text{diag}(w_1, \ldots, w_N)$;
\item \label{LEMRES2} The solution $J$ to \eqref{Flux3} obeys $|J|^2 \leq \frac{2\, N \, |\bar{v}^0|_{\infty}}{f_0} \, (-J \, : \,  (\nabla \mu-b))$.
\end{enumerate}
\end{lemma}
\begin{proof}
\eqref{LEMRES1}: For $\rho \in \mathbb{R}^N_+$, the matrix $\tau := B(\rho) \, R$ is clearly symmetric. Moreover, one can directly compute that $\tau \, 1^N = 0$. For $z \in \mathbb{R}^N$ arbitrary, one has
\begin{align}\label{tauprod}
\sum_{i,k = 1}^N \tau_{i,k} \, z_i \, z_k = \frac{\varrho}{2} \, \sum_{i\neq k} f_{i,k} \, y_i \, y_k \, (z_i - z_k)^2 \, .
\end{align}
Since the $f_{i,k}$ are strictly positive, \eqref{tauprod} shows that $\tau = \tau(\rho)$ is positive semi-definite on all positive states $\rho \in \mathbb{R}^N_+$. Thus, the matrix $R^{-1} \, B(\rho) = R^{-1} \, \tau(\rho) \, R^{-1}$ is likewise symmetric, positive semi-definite and possesses $\{\rho\}$ as its kernel.

For $\alpha > 0$, we introduce the matrix $B_{\alpha}(\rho) := B(\rho) + \alpha \, y \otimes 1^N$ which is invertible and inverse positive for small $\alpha >0$ (see \cite{bothedruetMS}). Then $R^{-1} \, B_{\alpha} = R^{-1} \, B + \frac{\alpha}{\varrho} \, 1^N\otimes 1^N$ is symmetric and strictly positive definite.

We next consider $\rho = X(s, \, w)$ where $X$ is the solution map for \eqref{ODE2}, $s = \hat{p}(\rho)$ and $w \in S^0$. We denote $W = \text{diag}(w_1, \ldots, w_N)$, and consider a matrix $K := W^{\frac{1}{2}} \, R^{-1} \, B_{\alpha} \, W^{\frac{1}{2}}$ which is positive definite, symmetric, and possesses the entries
\begin{align*}
 K_{i,j} = \begin{cases}
            (-f_{i,j}+\alpha) \, \frac{\sqrt{w_i \, w_j}}{\varrho} & \text{ for } i \neq j\\
            \frac{w_i}{\varrho} \, (\frac{1}{\rho_i} \, \sum_{k\neq i} f_{i,k} \, \rho_k + \alpha) & \text{ for } i = j  
           \end{cases} \, .
\end{align*}
Recall that $\rho_i = X_i(s, \, w)$ implies that $\rho_i/w_i = F_i(s, \, w) = F_i$ with $F$ defined in \eqref{Ffirst}. We now make use of the structural assumption (B3). For $s \in ]0, \, p_1[$ and $x \in S^1_+$, we might express $f_{i,j}(s, \, x) = f^{\text{reg}}_{i,j}(s, \, x) \, \sigma_{i,j}(s, \, x)$, with $f^{\text{reg}}$ uniformly bounded from below and above by positive constants $f_0 < f_1$, and the $\sigma_{i,j}(s, \, x)$ are by definition (cp. \eqref{asympfric}, \eqref{meaningsmallg})
\begin{align*}
\sigma_{i,j}(s, \, x) = \frac{F_i(s, \, w) \, F_j(s, \, w)}{F(s, \, w) \cdot \hat{y}(X(s, \, w))} \, .
\end{align*}
For $i \neq j$, we use $F_i \, \sqrt{w_i} = \sqrt{F_i \, \rho_i}$ and $F \cdot \hat{y} \, \varrho = F \cdot \rho$, and we therefore obtain that
\begin{align*}
 K_{i,j} = & (-f_{i,j}+\alpha) \, \frac{\sqrt{w_i \, w_j}}{\varrho} \\
 = & -f^{\text{reg}}_{i,j} \, \frac{F_i \, F_j}{F \cdot \hat{y}} \, \frac{\sqrt{w_i \, w_j}}{\varrho} +\alpha \, \frac{\sqrt{w_i \, w_j}}{\varrho}\\
 = & -f^{\text{reg}}_{i,j} \, \frac{\sqrt{F_i \, \rho_i} \, \sqrt{F_j \, \rho_j}}{F \cdot \rho} \, +\alpha \, \frac{\sqrt{w_i \, w_j}}{\varrho} \,.
 \end{align*}
Hence $|K_{i,j}| \leq f_1  + \alpha \, \frac{|w|_{\infty}}{\varrho}$. For $j = i$, we use (B3) again and we have
\begin{align*}
0 \leq  K_{i,i} = & \frac{1}{F_i} \, \sum_{k\neq i} f_{i,k} \, y_k + \alpha \, \frac{w_i}{\varrho} \\
  =  & \sum_{k\neq i} f^{\text{reg}}_{i,k} \,\frac{F_k \, y_k}{F \cdot y} + \alpha \, \frac{w_i}{\varrho}\\
  \leq & f_1 + \alpha \, \frac{|w|_{\infty}}{\varrho} \, .
\end{align*}
We therefore can bound $\max_{i,j} |K_{i,j}| \leq f_1 + \alpha \, |w|_{\infty}/\varrho$.

Thus the spectral radius $r(K)$ is bounded by the same quantity, and
\begin{align*}
 \lambda_{\min}(K^{-1}) \geq \frac{1}{f_1 + \alpha \, |w|_{\infty}/\varrho} =: d(\alpha) \, .
\end{align*}
In other words, the matrix $K^{-1} - d(\alpha) \, I$ is positive semi-definite, which means that
\begin{align*}
 W^{-\frac{1}{2}} \, B_{\alpha}^{-1} \, R \, W^{-\frac{1}{2}} - d(\alpha) \, I \geq 0 \, .
\end{align*}
We multiply from the left with the matrix $P^{\sf T} \, W^{\frac{1}{2}}$ and from the right with $W^{\frac{1}{2}} \, P$, preserving the inequality, and thus $P^{\sf T} \, B_{\alpha}^{-1} \, R \, P \geq d(\alpha) \, P^{\sf T} \, W \, P$. Recall that $B_{\alpha}^{-1} = B^D + \frac{1}{\alpha} \, y \otimes 1^N$ and $P^{\sf T} \, y = 0$, to see that $P^{\sf T} \, B^D \, R \, P \geq d(\alpha) \, P^{\sf T} \, W \, P$. Finally, $B^D \, R \, P = B^D \, R$ and $P^{\sf T} \, B^D = B^D$ imply that $B^D \, R \geq d(\alpha) \, P^{\sf T} \, W \, P$.

Thus, we can let $\alpha \rightarrow 0$, obtaining the claim \eqref{LEMRES1}.

We next prove \eqref{LEMRES2}. Due to the fact that $B(\rho) \, \hat{y}(\rho) = 0$, we can restate \eqref{Flux3} in the form  $B(\rho) \, \tilde{J} = - d$ with $\tilde{J} := J - \frac{J \cdot F}{y \cdot F} \, y$. Notice that $\tilde{J}$ is orthogonal to the vector $F = F(s, \, w)$ related to $\rho$ via \eqref{Ffirst} whenever $\rho = X(s, \, w)$. 

With $\tau := B(\rho) \, R$, we then have
\begin{align}\label{Flux3pr}
\tau \, R^{-1}\tilde{J} = - R \, P(y) \, (\nabla \mu - b) \, .
\end{align}
For fixed $\ell \in \{1,2,3\}$, consider the row for $\tilde{J}_{\ell} \in \mathbb{R}^N$ in the latter system. We multiply in \eqref{Flux3pr} from the right with $R^{-1} \tilde{J}_{\ell}$. Recalling that $P(y)^{\sf T} \, y = 0$, note that
\begin{align*}
- R \, P(y) \, (\partial_{x_{\ell}} \mu - b_{\ell}) \cdot R^{-1} \tilde{J}_{\ell} = -(\partial_{x_{\ell}} \mu - b_{\ell}) \cdot P(y)^{\sf T} \tilde{J}_{\ell} = -(\partial_{x_{\ell}} \mu - b_{\ell}) \cdot J_{\ell} \, .
\end{align*}
Thus, invoking \eqref{tauprod} and (B3), 
\begin{align*}
 \zeta = & - \sum_{\ell=1}^3 (\partial_{x_{\ell}} \mu - b_{\ell}) \cdot \tilde{J}_{\ell} = \sum_{\ell=1}^3 \tau \, R^{-1}\tilde{J}_{\ell} \cdot R^{-1}\tilde{J}_{\ell} \\
 = & \frac{\varrho}{2} \, \sum_{\ell=1}^3\sum_{i\neq k} f_{i,k} \, y_i \, y_k \, (\frac{\tilde{J}^i_{\ell}}{\rho_i} - \frac{\tilde{J}^k_{\ell}}{\rho_k})^{2}\\
 = & \frac{\varrho}{2} \, \sum_{\ell=1}^3\sum_{i\neq k} f^{\text{reg}}_{i,k} \, F_i \, y_i \, \frac{F_k \, y_k}{F \cdot y} \, (\frac{\tilde{J}^i_{\ell}}{\rho_i} - \frac{\tilde{J}^k_{\ell}}{\rho_k})^{2} \, .
\end{align*}
The numbers $\frac{F_k \, y_k}{F \cdot y}$ are positive and they sum up to one. Using convexity and the fact that $F \cdot \tilde{J} = 0$ by construction, we thus obtain that
\begin{align}\label{zwischenstep}
 \zeta & \geq  f_0 \, \frac{\varrho}{2} \, \sum_{\ell=1}^3\sum_{i,k = 1}^N  \, F_i \, y_i \, \frac{F_k \, y_k}{F \cdot y} \, ( \frac{\tilde{J}^i_{\ell}}{\rho_i} - \frac{\tilde{J}^k_{\ell}}{\rho_k})^{2} \nonumber \\
& \stackrel{\text{convex}}{\geq} f_0 \, \frac{\varrho}{2} \, \sum_{\ell=1}^3\sum_{i=1}^N  \, F_i \, y_i \, \big(\sum_{k = 1}^N\frac{F_k \, y_k}{F \cdot y} \, (\frac{\tilde{J}^i_{\ell}}{\rho_i} - \frac{\tilde{J}^k_{\ell}}{\rho_k}  ) \big)^2\nonumber\\
 & = f_0 \, \frac{\varrho}{2} \, \sum_{\ell=1}^3\sum_{i=1}^N  \, F_i \, y_i \, \big(\frac{\tilde{J}^i_{\ell}}{\rho_i}\big)^2 =  \frac{ f_0}{2} \, \sum_{\ell=1}^3\sum_{i=1}^N  \, \frac{1}{w_i} \, \big(\tilde{J}^i_{\ell}\big)^2 \, .
\end{align}
Due to H\"older's inequality, $|\tilde{J}_{\ell}|_1 \leq |w|_1^{1/2} \, |\frac{\tilde{J}_{\ell}}{\sqrt{w}}|_2$. Using the uniform bound for $w$, we get
\begin{align*}
\sum_{\ell = 1}^3 |\tilde{J}_{\ell}|^2_1 \leq \frac{2\, N \, |\bar{v}^0|_{\infty}}{f_0} \, \zeta \, . 
\end{align*}
It remains to recall that $\tilde{J} = J - \frac{J \cdot F}{y \cdot F} \, y$. Since $J \cdot 1^N = 0$, we must have $J = \tilde{J} - y \, \tilde{J} \cdot 1^N$, so that $J$ satisfies the bound claimed in \eqref{LEMRES2}.
\end{proof}
Lemma \ref{ROBUSTLEM} provides the needed estimates at small pressure. Of course, for normal pressure, we have also the expected behaviour, as stated in the next remark.
\begin{remark} \label{ROBUSTREM}
For $\rho = X(s, \, w)$ with $s = \hat{p}(\rho) > p_1$ and $w \in S_0$, the assumption (B4) for the Maxwell-Stefan coefficients yield $\min_{i\neq k} f_{i,k}(\rho) \geq f_2 > 0$ and $\max_{i\neq k} f_{i,k} \leq f_3$. Under these conditions we obtain instead of the estimates \eqref{LEMRES1}, \eqref{LEMRES2} of Lemma \ref{ROBUSTLEM} the normal Maxwell-Stefan bounds, and we deduce that 
 \begin{enumerate}[(1)]
 \addtocounter{enumi}{+2}
 \item \label{LEMRES1pr} The inverse Maxwell-Stefan matrix satisfies $(B(\rho))^D \, R \geq \frac{c}{f_3} \, P^{\sf T}(\hat{y}(\rho)) \, W \, P^{\sf T}(\hat{y}(\rho))$ where $W = \text{diag}(w_1, \ldots, w_N)$;
\item \label{LEMRES2pr} The solution $J$ to \eqref{Flux3} obeys $\sum_{i=1}^N |J^i|^2 \leq \frac{2}{f_2} \, |\rho| \,  (-J \, : \,  (\nabla \mu-b))$.
\end{enumerate}
\end{remark}
\begin{proof}
 In order to prove \eqref{LEMRES1pr}, we replace the matrix $K$ in the proof of Lemma \ref{ROBUSTLEM} with $\tilde{K} := R^{\frac{1}{2}} \, R^{-1} \, B_{\alpha} \, R^{\frac{1}{2}}$, and show for the spectral radius that $r(\tilde{K}) \leq f_3 + \alpha/\varrho$. Arguing analogously, we obtain first that
  \begin{align*}
P^{\sf T} \, B_{\alpha}^{-1} \, R \, P \geq d(\alpha) \, P^{\sf T} \, R \, P \, \quad  d(\alpha) := \frac{1}{f_3 + \alpha/\varrho} \, ,
\end{align*}
 and letting $\alpha \rightarrow 0$, it follows that $B^D \, R \geq \frac{1}{f_3} \, P^{\sf T} \, R \, P$. Recall that $\rho_i = F_i(s, \, w) \, w_i$. Thus $R = \text{diag}(F_1(s, \, w), \ldots,F_N(s, \, w)) \, W$. Since $F_i(s, \, w) \geq \underline{F}_i(s)$ according to Lemma \ref{characteristics}, \eqref{Quotients}, we obtain for the range $s = \hat{p}(\rho) \geq p_1$ that $R \geq \inf_{i = 1,\ldots,N, \, s \geq p_1} \underline{F}_i(s) \, W$, proving \eqref{LEMRES1pr}.
 
 In order to prove \eqref{LEMRES2pr}, we consider $\tau := B(\rho) \, R$. Hence, as in \eqref{Flux3pr}, $\tau \, R^{-1} J  \cdot R^{-1} J = - J \,:\, (\nabla \mu - b)$. The left-hand is nothing else but
\begin{align*}
 \tau \, R^{-1} J  \cdot R^{-1} J = &  \frac{\varrho}{2} \, \sum_{\ell=1}^3\sum_{i\neq k} f_{i,k} \, y_i \, y_k \, (\frac{J^i_{\ell}}{\rho_i} - \frac{J^k_{\ell}}{\rho_k})^{2}\\
 \geq & f_2 \, \frac{\varrho}{2} \, \sum_{\ell=1}^3\sum_{i = 1}^N y_i \, (\frac{J^i_{\ell}}{\rho_i})^2 = \frac{f_2}{2} \, \sum_{\ell=1}^3\sum_{i = 1}^N \frac{1}{\rho_i} \, (J^i_{\ell})^2 \, . 
\end{align*}
Thanks to H\"older's inequality, we then find that
\begin{align*}
 |J|^2 \leq |\rho| \, \sum_{\ell=1}^3\sum_{i = 1}^N \frac{1}{\rho_i} \, (J^i_{\ell})^2 \leq \frac{2}{f_2} \, |\rho| \, (- J \, : \, (\nabla \mu - b)) \, .
 \end{align*}
\end{proof}



\section{Proof of the main result}

Due to several prior investigations, it is not necessary to proof the result of Th. \ref{MAIN} in every detail. We shall present the general sketch putting some additional emphasis only on the main steps and the new ideas.

\paragraph{Approximation.} We construct approximate solutions using entropic variables and a stabilisation of the diffusion fluxes. For $\sigma > 0$, consider a regularised kinetic matrix
\begin{align}\label{Msigma}
M^{\sigma}(\rho) := B^D(\rho) \, R + \sigma \, {\rm I} \, ,
\end{align}
where $B^D(\rho) \, R$, $R := \text{diag}(\rho)$ is the matrix resulting from inversion of the Maxwell-Stefan matrix (cf. Lemma \ref{maxstefinverse}). 

We further denote $h$ the free energy function \eqref{FE}. As shown in \cite{dredrugagu17a}, this function is of Legendre type on $\mathbb{R}^N_+$ and the convex conjugate $h^*$ is of Legendre type on $\mathbb{R}^N$. Our approximation scheme consists in solving the system for variables $\mu^{\sigma} = (\mu_1^{\sigma},\ldots,\mu_N^{\sigma})$ and $v^{\sigma} = (v_1^{\sigma},v_2^{\sigma},v_3^{\sigma})$ the PDEs
\begin{align}
 \label{masssigma}\partial_t \rho^{\sigma}_i + \divv( \rho^{\sigma}_i \, v^{\sigma} -\sum_{j=1}^N M^{\sigma}_{i,j}(\rho^{\sigma}) \, (\nabla \mu^{\sigma}_j - b^j)) & = 0 \qquad \text{ for } i = 1,\ldots,N\\
\label{momentumsigma}\partial_t (\varrho^{\sigma} \, v^{\sigma}) + \divv( \varrho^{\sigma} \, v^{\sigma}\otimes v^{\sigma} - \mathbb{S}(\nabla v^{\sigma})) + \nabla p^{\sigma} & = \sum_{i=1}^N \rho^{\sigma}_i \, b^i + \divv (\sum_{i=1}^N \jmath^{\sigma,i} \,  v)   \, .
\end{align}
subject to the identities 
\begin{align}\label{auxiliaries}
\rho^{\sigma}_i := \partial_{\mu_i} h^*(\mu^{\sigma}), \quad p^{\sigma} := h^*(\mu^{\sigma}), \quad \jmath^{\sigma} := - M^{\sigma}(\rho^{\sigma}) \, (\nabla \mu^{\sigma} - b) \, .
\end{align}
We consider the initial conditions \eqref{initial0v} for $v$, and for $\mu$
\begin{alignat}{2}\label{initialsigmarho}
\mu_i^{\sigma}(x,\, 0) & = \partial_{\rho_i}h(\rho^0(x)) & & \text{ for } x \in \Omega, \, i = 1,\ldots, N \, ,
\end{alignat}
with $\rho^0$ from \eqref{initial0rho}. On $S_T$ we assume \eqref{lateral0v} and $\nu \cdot \jmath^{\sigma} = 0$. 

Due to the strict convexity of $h^*$, the approximation scheme \eqref{masssigma}, \eqref{momentumsigma} constitutes a nonlinear parabolic system for $(\mu, \, v)$.
There are no algebraic constraints on the unknowns. Existence for the regularised problem can be proved as in \cite{dredrugagu16}, \cite{dredrugagu17b} by means of a Galerkin approximation scheme. The weak solution possesses the regularity
\begin{align}
\mu^{\sigma} \in W^{1,0}_2(Q_T; \, \mathbb{R}^N), \, v^{\sigma} \in W^{1,0}_2(Q_T; \, \mathbb{R}^3) \, ,
\end{align}
and it satisfies for all $t \leq T$ the energy dissipation inequality
\begin{align}\label{DISS}
\int_{\Omega \times \{t\}} \{h(\rho^{\sigma}) + \frac{\varrho^{\sigma}}{2} \, |v^{\sigma}|^2\} \, dx + \int_{Q_t} \{M^{\sigma}(\rho^{\sigma}) \, (\nabla \mu^{\sigma} - b)  :  (\nabla \mu^{\sigma} - b) + \mathbb{S}(\nabla v^{\sigma}) : \nabla v^{\sigma}\} \, dxd\tau \nonumber \\
\leq \int_{\Omega} \{h(\rho^0) + \frac{\varrho_0}{2} \, |v^0(x)|^2\} \, dx - \int_{Q_t} M^{\sigma}(\rho^{\sigma}) \, (\nabla \mu^{\sigma} - b) :   b \, dxd\tau + \int_{Q_t} \rho^{\sigma} \, b\cdot v^{\sigma}  \, dxd\tau \, .
\end{align}
After estimating the right-hand side by means of H\"older's and Young's inequalities, we attain the inequality
\begin{align}\label{DISS2}
\int_{\Omega\times\{t\}} \{h(\rho^{\sigma}) + \frac{\varrho^{\sigma}}{2} \, |v^{\sigma}|^2\} \, dx + \int_{Q_t} \{\frac{1}{2} \, M^{\sigma}(\rho^{\sigma}) \, (\nabla \mu^{\sigma} - b) \, : \,  (\nabla \mu^{\sigma} - b) + \mathbb{S} \, : \, \nabla v^{\sigma}\} \, dxd\tau \nonumber \\
\leq E_0 + \frac{\|b\|_{L^{\infty}(Q_T)}^2}{2} \int_{Q_t} \{|M^{\sigma}(\rho^{\sigma})| + |\rho^{\sigma}|\} \, dxd\tau + \int_{Q_t} \frac{\varrho^{\sigma}}{2} \, |v^{\sigma}|^2 \, dxd\tau \, .
\end{align}
Recall that $M^{\sigma} = B^D \, R + \sigma \, I$ with $B^D \, R$ resulting from inversion of the Maxwell-Stefan equations.
The Remark \ref{ROBUSTREM}, \eqref{LEMRES2pr} shows that $|B^D(\rho) \, R| \leq  C \, |\rho|$ if $\hat{p}(\rho )$ is sufficiently large.
Thus, we can obtain the linear-growth estimate $|M^{\sigma}(\rho)| \leq |B^D(\rho) \, R| + \sigma \leq C \, (1+|\rho|)$.
Showing next that $h$ growth at least linearly, we can apply the Gronwall Lemma to \eqref{DISS2}, and we see that the left-hand is uniformly bounded.
\begin{lemma}\label{growthofh}
Consider the function $h$ of \eqref{FE}. Assume that the functions $g_i$ therein satisfy {\rm (A5)}, and define $\alpha := \max\{\alpha_1, \ldots, \alpha_N\} > 1$. Then there are $c_0 > 0$ and $c_1 > 0$ such that $h(\rho) \geq c_0 \, |\rho|^{\frac{\alpha}{\alpha-1}} - c_1$.
\end{lemma}
\begin{proof}
We define $H(\rho) := \sum_{i=1}^N \rho_i \, g_i(\hat{p}(\rho)) - \hat{p}(\rho)$. Then $h(\rho) = H(\rho) + k(\rho)$ with the function $k$ of \eqref{entropicpart}. This $k$ is non-positive, and it is readily checked that $|k(\rho)| \leq C_k \, |\rho|_1$. We recall the side-condition $\sum_{i=1}^N \rho_i \, g_i^{\prime}(\hat{p}(\rho)) = 1$ (cf. \eqref{PRESSURESTATE}). We pass to the variables $s, \, w$ of section \ref{PARAM}, we invoke the assumption (A5), and we verify for $s > M_1$ that
\begin{align}\label{H1}
H(X(s, \, w)) = & \sum_{i=1}^N X_i(s, \, w) \, g_i(s) - s \geq s \, (\beta \, \sum_{i=1}^N  X_i(s, \, w) \, g^{\prime}_i(s) - 1) = (\beta - 1) \, s \, .
\end{align}
Next, the assumption (A5) also guarantees that $g_i \leq \alpha_i \, g_i^{\prime} \,s $. We integrate over $[M_1, \, s]$ to see that $g_i(s) \geq \frac{g_i(M_1)}{M_1^{1/\alpha_i}} \, s^{\frac{1}{\alpha_i}}$. Thus, there is a positive constant $C$ depending only on $\alpha$ and $M_1$ such that $\min g^{\prime}(s) \geq C \, (1/s)^{1-\frac{1}{\max \alpha}}$. Recalling Lemma \ref{characteristics}, \eqref{Xtot}, we find that $|X(s, \, w)|_1 \leq  (\min g^{\prime}(s))^{-1} \leq k_0^{-1} \, s^{1-\frac{1}{\max \alpha}}$.
Combining this with \eqref{H1}, we find that $H(\rho) \geq c_0 \, |\rho|_1^{\frac{\max \alpha}{\max \alpha - 1}}$ if $\hat{p}(\rho) \geq M_1$. Thus, $$h(\rho) \geq c_0 \,|\rho|_1^{\frac{\max \alpha}{\max \alpha - 1}} - \sup_{\hat{p}(\rho) \leq M_1} |H(\rho)| - C_k \,   |\rho|_1 \, ,$$ and the claim follows.
 \end{proof}
 Notice that, while estimating $\nabla \mu^{\sigma}$ in $L^2(Q_T)$ is a straightforward consequence of \eqref{DISS2} and of the choice of $M^{\sigma} \geq \sigma \, {\rm I}$, the control on $\|\mu^{\sigma}\|_{L^2(Q_T)}$ results from more subtle arguments. Indeed, all contributions of $\mu^{\sigma}$ to the dissipation are due to the gradient. To obtain a bound, we must in particular exploit the fact that $\int_{\Omega} \rho_i^{\sigma}(x, \, t) \, dx$ is a conserved quantity for each $i$. We refer to \cite{dredrugagu17b} for details.
 
From \eqref{DISS} and Lemma \ref{growthofh}, we can motivate the regularity $\rho^{\sigma} \in L^{\gamma,\infty}(Q_T; \, \mathbb{R}^N)$ and $\sqrt{\varrho^{\sigma}} \, v^{\sigma} \in L^{2,\infty}(Q_T; \, \mathbb{R}^3)$. Using \eqref{auxiliaries}, the Gibbs-Duhem equation $\nabla p^{\sigma} = \sum_{i=1}^N \rho_i^{\sigma} \, \nabla \mu_i^{\sigma}$ results naturally. Hence $\|\nabla p^{\sigma}\|_{L^{2\gamma/(2+\gamma),2}(Q_T)} \leq \|\rho^{\sigma}\|_{L^{\gamma,\infty}(Q_T)} \, \|\nabla \mu^{\sigma}\|_{L^2(Q_T)}$. With analogous arguments: $\jmath^{\sigma} \in L^{2\gamma/(2+\gamma),2}(Q_T)$.

\paragraph{Uniform estimates.} Several uniform estimates are standard consequences of the energy inequality \eqref{DISS}, so we will state them without comment. We again refer to \cite{dredrugagu17b} for details. Notice moreover that among these estimates, several are in principle the same as for single component Navier-Stokes equations:
\begin{align*}
\|\rho^{\sigma}\|_{L^{\gamma,\infty}(Q_T; \, \mathbb{R}^N)} + \|\sqrt{\varrho^{\sigma}} \, v^{\sigma}\|_{L^{2,\infty}(Q_T; \, \mathbb{R}^3)} + \|v^{\sigma}\|_{W^{1,0}_2(Q_T; \, \mathbb{R}^3)} \leq C_0 \, ,\\
\|\varrho^{\sigma} \, v\|_{L^{\frac{6\gamma}{6+\gamma},2}(Q_T, \, \mathbb{R}^3)} + \|\varrho^{\sigma} \, v\|_{L^{\frac{2\gamma}{1+\gamma},\infty}(Q_T, \, \mathbb{R}^3)} \leq C_0\, ,\\
\|\varrho^{\sigma} \, |v^{\sigma}|^2\|_{L^{\frac{3\gamma}{3+\gamma},1}(Q_T)} + \|\varrho^{\sigma} \, |v^{\sigma}|^2\|_{L^{\frac{5}{3} -\frac{1}{\gamma}}(Q_T)} \leq C_0 \, , \\
\|\sum_{i=1}^N \jmath^{\sigma,i}\|_{L^2(Q_T; \, \mathbb{R}^3)} + \|\sum_{i=1}^N \jmath^{\sigma,i} \otimes v^{\sigma}\|_{L^{\frac{3}{2},1}(Q_T; \, \mathbb{R}^{3\times 3})} \leq C_0 \, \sqrt{\sigma} \, .
\end{align*}
Restricting for simplicity to the case $\gamma > 3$ as announced, we moreover can uniformly estimate the approximate pressures
\begin{align*}
\|p^{\sigma}\|_{L^{1+\frac{1}{\gamma}}(Q_T)} \leq C_0 \, .
\end{align*}
Only the uniform estimate for $\jmath^{\sigma}$ is missing. Since $\jmath^{\sigma} = - M^{\sigma}(\rho^{\sigma}) \, (\nabla \mu^{\sigma} - b)$, we can decompose 
\begin{align}\label{fluxdecompo}
\jmath^{\sigma} = & - B^D(\rho^{\sigma}) \, R^{\sigma} \, (\nabla \mu^{\sigma}-b) -  \sigma \, (\nabla \mu^{\sigma}-b) =: J^{\sigma} + \tilde{\jmath}^{\sigma} \, ,
\end{align}
where we abbreviated $R^{\sigma} := \text{diag}(\rho^{\sigma}_1, \ldots, \rho^{\sigma}_N)$. By the definition of $B^D$ (see Lemma \ref{maxstefinverse}), the first contribution $J^{\sigma}$ solves the Maxwell-Stefan equations
\begin{align}\label{NULL}
B(\rho^{\sigma}) \, J^{\sigma} = - R^{\sigma} \, P(\hat{y}(\rho^{\sigma})) \, (\nabla \mu^{\sigma}-b) \, .
\end{align}
To obtain estimates, we next decompose the domain $Q_T$ between points for which $p^{\sigma}(x, \, t) \leq p_1$ and for which $p^{\sigma}(x, \, t) > p_1$, where $p_1$ is the constant occurring in the condition (B3).

For $p^{\sigma}(x,t) \leq p_1$, the Lemma \ref{ROBUSTLEM} implies $|J^{\sigma}(x,t)|^2 \leq C \, (-J^{\sigma} : (\nabla \mu^{\sigma}-b))$. Since \eqref{DISS2} implies a bound in $L^1(Q_T)$ for $-J^{\sigma} : (\nabla \mu^{\sigma}-b)$, we thus control the integral of $|J^{\sigma}|^2$ over this set.

If $p^{\sigma}(x,t) \geq p_1$, we obtain the standard Maxwell-Stefan bound (see Remark \ref{ROBUSTREM}), which means that $|J^{\sigma}(x,t)|^2 \leq C \, |\rho^{\sigma}| \,  (-J^{\sigma} : (\nabla \mu^{\sigma}-b))$. We combine both uniform bounds, and they yield
\begin{align*}
\|J^{\sigma}\|_{L^{\frac{2\gamma}{1+\gamma},2}(Q_T; \, \mathbb{R}^{N\times 3})} \leq C_0 \, .
\end{align*}
Passing to variables $\rho^{\sigma} = X(p^{\sigma}, \, w^{\sigma})$ as described in Lemma \ref{characteristics}, we obtain via Lemma \ref{ROBUSTLEM} and Remark \ref{ROBUSTREM} the additional bound
\begin{align}\label{robinsitu}
\|\nabla w^{\sigma}\|_{L^{2}(Q_T; \, \mathbb{R}^{N\times 3})} \leq C_0 \, .
\end{align}
This was demonstrated in the section \ref{MAINIDEAS}. It is readily seen that the second contribution in \eqref{fluxdecompo} obeys $\|\tilde{\jmath}^{\sigma}\|_{L^2(Q_T)} \leq C_0 \, \sqrt{\sigma}$.

\paragraph{Limit passage.} Now we possess enough informations to extract weak limits $(\rho, \, v, \, p, \, J)$ for $\{\rho^{\sigma}\}$, $\{v^{\sigma}\}$, $\{p^{\sigma}\}$ and $\{J^{\sigma}\}$. 
 
We show in the canonical manner that $v^{\sigma} \rightarrow v$ strongly in $L^2(Q_T; \, \mathbb{R}^3)$, that $\varrho^{\sigma} \, v^{\sigma} \rightarrow \varrho \, v$ weakly in $L^1(Q_T; \, \mathbb{R}^3)$ and that $\varrho^{\sigma} \, v^{\sigma} \otimes v^{\sigma} \rightarrow \varrho \, v \otimes v$ weakly in $L^1(Q_T; \, \mathbb{R}^{3\times 3})$. These are well known structural consequences for the Navier-Stokes operator (cp. \cite{lionsfils}).

Passing to variables $\rho^{\sigma} = X(p^{\sigma}, \, w^{\sigma})$, we have $p^{\sigma} = \mathscr{P}(\varrho^{\sigma}, \, w^{\sigma})$, and the Lions method provides the strong convergence of the mass densities $\varrho^{\sigma} \rightarrow \varrho$ in $L^1(Q_T)$ -- as demonstrated in section \ref{PQ}.

Once this point is clear, we can rephrase $\rho^{\sigma} = \mathscr{R}(\varrho^{\sigma}, \, w^{\sigma})$ with a certain mapping $\mathscr{R}$, and we obtain the compactness of the whole vector of mass densities as for instance in \cite{dredrugagu17c}. Due to the bijectivity of the variable change, the auxiliary variables $w^{\sigma}$ also converge strongly, for instance in $L^1(Q_T; \, \mathbb{R}^N)$.
These properties suffice for passing to the limit in the integral relations resulting from the distributional formulations of \eqref{masssigma} and \eqref{momentumsigma}.

We will focus on passing to the limit in the approximate Maxwell-Stefan equations \eqref{NULL}. 

The right-hand side $d^{\sigma} := R^{\sigma} \, P(\hat{y}(\rho^{\sigma})) \, (\nabla \mu^{\sigma}-b)$ of \eqref{NULL} possesses the componentwise expression
\begin{align}\label{NULL1}
d^{\sigma,i} = \rho^{\sigma}_i \, \sum_{j=1}^N (\delta_{i,j} - \hat{y}_j(\rho^{\sigma})) \, (\nabla \mu^{\sigma}_j-b^j) \, .
\end{align}
Recalling \eqref{auxiliaries}, we have $\rho^{\sigma} = \nabla_{\mu} h^*(\mu^{\sigma})$, hence also $\mu^{\sigma} = \nabla_{\rho} h(\rho^{\sigma}) = \hat{\mu}(\hat{p}(\rho^{\sigma}), \, \hat{x}(\rho^{\sigma}))$. 

We pass to the variables $p^{\sigma}, \, w^{\sigma}$ for which $\rho^{\sigma} = X(p^{\sigma}, \, w^{\sigma})$. Then, exploiting \eqref{Proper} once again, we find for $\eta \in \{1^N\}^{\perp}$ arbitrary that
\begin{align*}
 \eta \cdot \mu^{\sigma} = \eta \cdot \nabla_{\rho} h(X(p^{\sigma}, \, w^{\sigma})) = \eta \cdot \nabla_{\rho} h(w^{\sigma}) \, .
\end{align*}
As $\mu^{\sigma} \in W^{1,0}_2(Q_T; \, \mathbb{R}^N)$, computing the spatial gradients in the latter identity is possible. It follows that
$ \eta \cdot \nabla \mu^{\sigma} = \eta \cdot D^2h(w^{\sigma}) \,  \nabla w^{\sigma}$. Thus, the $d^{\sigma}$ of \eqref{NULL1} possess the following equivalent representation:
\begin{align}\label{NULL2}
d^{\sigma,i} = \rho^{\sigma}_i \, \sum_{j=1}^N (\delta_{i,j} - \hat{y}_j(\rho^{\sigma})) \, (e^jD^2h(w^{\sigma}) \nabla w^{\sigma}-b^j) \, .
\end{align}
Let us consider in more details the right-hand of \eqref{NULL2}. First, the identity \eqref{HESS} implies that
\begin{align*}
 D^2_{j,k}h(w^{\sigma}) = \frac{\delta_{j,k}}{w^{\sigma}_j \, m_j} - \frac{1}{\hat{n}(w^{\sigma}) \, m_j\, m_k}  - \frac{g^{\prime}_j(p_0) \, g^{\prime}_k(p_0)}{ \sum_{\ell} g^{\prime\prime}(p_0) \, w^{\sigma}_\ell} \,,
 \end{align*}
 where we exploited that $w^{\sigma} \in S_0$ implies $\hat{p}(w^{\sigma}) = p_0$. For $i,k = 1,\ldots,N$, we define
 \begin{align}\label{matrixGamma} 
 \Gamma^{\sigma}_{i,k} := & \rho^{\sigma}_i \, \sum_{j=1}^N (\delta_{i,j} - \hat{y}_j(\rho^{\sigma})) \, D^2_{j,k}h(w^{\sigma})\nonumber\\
 = & \rho^{\sigma}_i \, \sum_{j=1}^N (\delta_{i,j} - \hat{y}_j(\rho^{\sigma})) \, (\frac{\delta_{j,k}}{w^{\sigma}_j \, m_j} - \frac{1}{\hat{n}(w^{\sigma}) \, m_j\, m_k}  - \frac{g^{\prime}_j(p_0) \, g^{\prime}_k(p_0)}{ \sum_{\ell} g^{\prime\prime}(p_0) \, w^{\sigma}_\ell})\nonumber\\
 = & \rho_{i}^{\sigma} \, \Big(\frac{\delta_{i,k}}{w^{\sigma}_i \, m_i} - \frac{1}{m_im_k \, \hat{n}(w^{\sigma})} - \frac{g^{\prime}_i(p_0) \, g^{\prime}_k(p_0)}{\sum_{\ell} g^{\prime\prime}(p_0) \, w^{\sigma}_\ell}\nonumber\\
  & \quad \quad \quad - \frac{\hat{y}_k(\rho^{\sigma})}{w^{\sigma}_k \, m_k} +  \frac{\hat{y}(\rho^{\sigma}) \cdot \frac{1}{m}}{\hat{n}(w^{\sigma}) \, m_k} + \frac{\hat{y}(\rho^{\sigma}) \cdot g^{\prime}(p_0) \, g_k^{\prime}(p_0)}{\sum_{\ell} g^{\prime\prime}(p_0) \, w^{\sigma}_\ell} \Big)\, .
  \end{align}
 Then, we can re-express the $d^{\sigma}$ in \eqref{NULL2} equivalently as
 \begin{align}\label{NULL3} 
 d^{\sigma,i} = \sum_{k=1}^N  \Gamma^{\sigma}_{i,k} \, \nabla w^{\sigma}_k - \rho^{\sigma}_i \, \sum_{j=1}^N (\delta_{i,j} - \hat{y}_j(\rho^{\sigma})) \, b^j \, .
\end{align}
From their definition in \eqref{matrixGamma}, the entries of the matrix $\Gamma^{\sigma}$ can be bounded via
\begin{align*} 
|  \Gamma^{\sigma}_{i,k}| \leq \frac{1}{\min m} \, (\frac{\rho^{\sigma}_i}{w^{\sigma}_i}+ \frac{\rho^{\sigma}_k}{w^{\sigma}_k}) + 2 \, (\frac{\max m}{\min m^2} + \frac{|g^{\prime}(p_0)|}{\min |g^{\prime\prime}(p_0)|}) \, \frac{\rho_i^{\sigma}}{w^{\sigma}\cdot 1^N} \, .
\end{align*}
For all $i,k$, the right-hand is a bounded function of the quotients $\rho^{\sigma}_i/w^{\sigma}_i = F_i(p^{\sigma}, \, w^{\sigma})$. From Lemma \ref{characteristics}, \eqref{Quotients}, we know that these quotients are bounded above by a function $\bar{F}$ (cp. \eqref{barF} for the explicit formula), which depends only on the first argument $p^{\sigma}$, and is continuous and bounded on compact sets of $]0,+\infty[$. Therefore, for all $\epsilon > 0$, there is $c(\epsilon)$ such that
\begin{align}\label{estimateGamma} 
 |\Gamma^{\sigma}(x, \, t)| \leq c(\epsilon) \text{ in } \{(x, \,t) \in Q_T \, : \, \epsilon \leq p^{\sigma}(x,t) \leq \epsilon^{-1}\} =: S^{\sigma}_{\epsilon} \, .
\end{align}
Next we prove for $\{\Gamma^{\sigma}\}$ some pointwise convergence. Recall that $\rho^{\sigma} \rightarrow \rho$ pointwise almost everywhere, and $w^{\sigma} \rightarrow w$ pointwise almost everywhere. Moreover, if $\rho_i(x, \, t) > 0$, then $w_i(x, \ t) > 0$: This is a consequence of the change of variables ( identity \eqref{representationw}). On the set
\begin{align*}
 A^+ := \{(x, \, t) \, : \, \rho_i(x, \, t) > 0 \text{ for } i = 1, \ldots,N\} \, ,
\end{align*}
we therefore find that $\Gamma^{\sigma}(x, \, t) \rightarrow \Gamma(x, \, t)$ for almost all $(x, \, t) \in A^+$, where $\Gamma(x,t)$ is the natural limit in \eqref{matrixGamma}. Now, we multiply the identity \eqref{NULL3} with $\phi^{\sigma,\epsilon} =$ indicator function of the set $A^+ \cap S^{\sigma}_{\epsilon}$ for $\epsilon$ fixed. In view of \eqref{estimateGamma}, the matrices $\phi^{\sigma,\epsilon} \, \Gamma^{\sigma}$ are bounded uniformly and pointwise convergent. Moreover, \eqref{robinsitu} allows us to assume that $\nabla w^{\sigma} \rightarrow \nabla w$ weakly in $L^2(Q_T)$. It follows from \eqref{NULL3} that
\begin{align}\label{troule}
  d^{\sigma,i} \, \phi^{\sigma,\epsilon} \rightarrow \phi^{0,\epsilon} \, (\sum_{k=1}^N  \Gamma_{i,k} \, \nabla w_k - \rho_i \, \sum_{j=1}^N (\delta_{i,j} - \hat{y}_j(\rho)) \, b^j) \,
 \end{align}
 in the sense of weak convergence in $L^2(Q_T)$. Here $\phi^{0,\epsilon}$ being the indicator function of the set $A^+ \cap S^0_{\epsilon}$ with $S^0_{\epsilon} := \{(x, \, t) \, : \, \epsilon \leq  p(x,t) \leq \epsilon^{-1}\}$. On the other hand, due to \eqref{NULL}, we have
 \begin{align*}
\phi^{\sigma,\epsilon} \,  B(\rho^{\sigma}) \, J^{\sigma} = \phi^{\sigma,\epsilon} \, ( - d^{\sigma,i} - B(\rho^{\sigma}) \, \tilde{\jmath}^{\sigma}) \, .
 \end{align*}
Since $\|\tilde{\jmath}^{\sigma}\|_{L^2} \leq C_0 \, \sqrt{\sigma}$, we pass to the limit in the latter relation. In the sense of weak convergence in $L^{2\gamma/(1+\gamma),2}(Q_T)$, we obtain that
 \begin{align}\label{boule} 
-    d^{\sigma,i} \, \phi^{\sigma,\epsilon} \rightarrow \phi^{0,\epsilon} \,  B(\rho) \, J \, .
\end{align}
Thus, comparing \eqref{troule} with \eqref{boule}, we obtain that
\begin{align}\label{maxstel} 
 B(\rho) \, J = -\sum_{k=1}^N  \Gamma_{i,k} \, \nabla w_k - \rho_i \, \sum_{j=1}^N (\delta_{i,j} - \hat{y}_j(\rho)) \, b^j
\end{align}
almost everywhere in $A^+ \cup S^0_{\epsilon}$. Now if $(x, \, t)$ is in $A^+$, then obviously $\sum_{i=1}^N \rho_i(x, \, t) > 0$ and $p(x, \, t) > 0$. Therefore, we see that $A^+ \subseteq \bigcup_{\epsilon > 0} S^0_{\epsilon}$, and indeed, \eqref{maxstel} holds true almost everywhere in $A^+$.

We would like to relate the right-hand of \eqref{maxstel} to the limit chemical potentials $\mu = \hat{\mu}(\hat{p}(\rho), \, \hat{x}(\rho))$.
By the construction of $\Gamma$ (see \eqref{matrixGamma}) we can at first rephrase \eqref{maxstel} as in the form \eqref{NULL2}
\begin{align}\label{NULL2pr} 
-(B(\rho) \, J)^i =  \rho_i \, \sum_{j=1}^N (\delta_{i,j} - \hat{y}_j(\rho)) \, (e^jD^2h(w) \nabla w-b^j) \, .
\end{align}
We next establish for our limit element $\rho$ the property \eqref{Condimu1} of weak solutions. For $\mu = \hat{\mu}(\hat{p}(\rho), \, \hat{x}(\rho)) = \nabla_{\rho} h(\rho)$, the trick \eqref{Proper} implies again that $\mathcal{P}_{\{1^N\}^{\perp}} \, \mu =  \mathcal{P}_{\{1^N\}^{\perp}}  \nabla_{\rho} h(w) $. Now, the vector $\nabla_{\rho} h(w)$ is, up to inessential constants, equal to $\ln \hat{x}^{\frac{1}{m}}(w)$, and there is always at least one among the fractions $\hat{x}_1(w), \ldots,\hat{x}_N(w)$ being larger than $N^{-1}$. Thus, boundedness of $\eta \cdot \ln \hat{x}^{\frac{1}{m}}(w)$ for all $\eta \in \{1^N\}^{\perp}$ implies the boundedness of all entries of $w$ from below.
With this argument, we can verify that for all $k > 1$, the truncations $[\mathcal{P}_{\{1^N\}^{\perp}} \, \nabla_{\rho} h(w)]^k$ are identical with $\mathcal{P}_{\{1^N\}^{\perp}} \, \nabla_{\rho} h([w]_{\ell(k)})$, where $[\cdot]_{\ell(k)}$ is some cut-off operator at small positive level. Such truncations are, with $w$ itself, elements of $W^{1,0}_2(Q_T; \, \mathbb{R}^N)$. Moreover, the map $r \mapsto \nabla_{\rho} h(r)$ is Lipschitz continuous in the subset of $\mathbb{R}^N_+$ such that $r_i \geq \ell(k) > 0$ for all $i$. Thus, we can verify $[\mathcal{P}_{\{1^N\}^{\perp}} \, \mu]^k \in W^{1,0}_2(Q_T; \, \mathbb{R}^N)$ for all $k$, proving the property \eqref{Condimu1}.

This allows to transform the relation \eqref{NULL2pr}. If $(x, \, t) \in A^+$, then we find $k$ large enough for which $\mu(x, \, t) = [\mu(x, \, t)]^k$ and $\mathcal{P}_{\{1^N\}^{\perp}} \, \mu = [\mathcal{P}_{\{1^N\}^{\perp}} \, \mu]^k$. We therefore can reinterpret the term $\sum_{j=1}^N (\delta_{i,j} - \hat{y}_j(\rho)) \, e^jD^2h(w) \nabla w$ in the right-hand of \eqref{NULL2pr} and this yields
\begin{align}\label{NULL3pr} 
-(B(\rho) \, J)^i = &  \rho_i \, \sum_{j=1}^N (\delta_{i,j} - \hat{y}_j(\rho)) \, (\nabla \mu_j-b^j) \, \nonumber\\
 =: & \rho_i \, \sum_{j=1}^N (\delta_{i,j} - \hat{y}_j(\rho)) \, (\nabla [\mathcal{P}_{\{1^N\}^{\perp}} \, \mu]_j^k-b^j)
\end{align}
This proves the first condition \eqref{maxstefweak1} for a weak solution of Maxwell-Stefan.

Finally we want to prove, as required in the second condition \eqref{maxstefweak2}, that $J^i = 0$ if $\rho_i = 0$ and $p(x, \, t) > 0$. Here we employ the intermediate result \eqref{zwischenstep} in the proof of Lemma \ref{ROBUSTLEM}. It shows for our approximations, that the quantities
$\sum_{\ell=1}^3\sum_{i=1}^N  \, \frac{1}{w_i^{\sigma}} \, \big(\tilde{J}^{\sigma,i}_{\ell}\big)^2$, where
\begin{align*}
 \tilde{J}^{\sigma} :=  J^{\sigma} - \frac{J^{\sigma} \cdot F(p^{\sigma},w^{\sigma})}{F(p^{\sigma},w^{\sigma})\cdot \hat{y}(\rho^{\sigma})} \, \hat{y}(\rho^{\sigma}) \, 
\end{align*}
are bounded above by the functions $\zeta_{\sigma} := - (\nabla \mu^{\sigma} - b) : J^{\sigma}$.
We multiply with the indicator functions $\phi^{\sigma,\epsilon}$. The sequence $ \tilde{J}^{\sigma} \,\phi^{\sigma,\epsilon}$ is again bounded in $L^{\frac{2\gamma}{1+\gamma},2}(Q_T)$. For its weak limit $\tilde{J}$, we obtain in the set $S^0_{\epsilon}$ the identity
\begin{align}\label{proust}
 \tilde{J} = J - \frac{J \cdot F(p, \, w)}{F(p, \, w)\cdot \hat{y}(\rho)} \, \hat{y}(\rho) \, .
\end{align}
As $(\tilde{J}^{\sigma,i})^2$ is bounded above by $\zeta^{\sigma}$ times $ w_i^{\sigma}$, we find that $\tilde{J} = 0$ for almost all $(x, \, t)$ in $S^0_{\epsilon}$ such that $w_i = 0$. But in the set $S^0_{\epsilon}$, we know that $w_i = 0$ is equivalent to $\rho_i = 0$. Hence, we can infer from \eqref{proust} that
\begin{align*}
 0 = J^i - \frac{J \cdot F(p, \, w)}{F(p, \, w)\cdot \hat{y}(\rho)} \, \hat{y}_i(\rho) = J^i \, 
\end{align*}
for almost all $(x, \, t)$ such that $\epsilon \leq p(x,t) \leq \epsilon^{-1}$ and $\rho_i(x,t) = 0$. This completes the proof that the weak solution $(\rho, \, v, \, J)$ also complies with the second condition \eqref{maxstefweak2}.

\section*{Acknowledgement}

I would like to thank Prof.~D.~Bothe for giving me insights into Maxwell-Stefan diffusion, and for the many fruitful discussions.

\end{document}